\documentclass[11pt]{amsart}
\usepackage{amssymb, amsmath, fullpage,graphicx,color,overpic}

\newcommand{\Inv}{\operatorname{Inv}}
\newcommand{\Fix}{\operatorname{Fix}}
\newcommand{\CAT}{\operatorname{CAT}}
\newcommand{\ii}{\mathbf{i}}
\newcommand{\jj}{\mathbf{j}}

\newcommand{\fin}{{\rm fin}}
\newcommand{\limW}{{\mathcal W}}
\newcommand{\bTX}{\partial_T X}
\newcommand{\bTXi}{\partial_T X(\mathbf{i})}
\newcommand{\bTXj}{\partial_T X(\mathbf{j})}

\newcommand{\bX}{\partial X}

\newcommand{\cbM}{\partial \mathring M}
\newcommand{\cM}{\overline{M}}
\newcommand{\bM}{\partial M}
\newcommand{\Mxy}{\mathcal{M}(x,y)}
\newcommand{\MF}{\mathcal{M}(F)}
\newcommand{\Mxi}{\mathcal{M}(\xi)}
\newcommand{\Rxi}{R(\xi)}
\newcommand{\Wxi}{W(\xi)}
\newcommand{\bMxi}{\partial\Mxi}

\newcommand{\cF}{\mathcal{F}}

\newcommand{\cR}{\mathcal{R}}
\newcommand{\cS}{\mathcal{S}}
\newcommand{\cT}{\mathcal{T}}

\newcommand{\R}{\mathbb{R}}
\newcommand{\Z}{\mathbb{Z}}

\newcommand{\Pc}{\operatorname{Pc}}

\newcommand{\s}{\sigma}

\newcommand{\la}{\langle}
\newcommand{\ra}{\rangle}

\newtheorem{theorem}{Theorem}
\newtheorem{proposition}[theorem]{Proposition}
\newtheorem{lemma}[theorem]{Lemma}
\newtheorem{corollary}[theorem]{Corollary}

\theoremstyle{definition}

\newtheorem{remark}{Remark}
\newtheorem{definition}{Definition}
\newtheorem{example}{Example}
\newtheorem*{examples}{Examples}

\title{Infinite reduced words \\and\\ the Tits boundary of a Coxeter group}
\author{Thomas Lam}\address{Department of Mathematics, University of Michigan,
2074 East Hall, 530 Church Street, Ann Arbor, MI 48109-1043, USA}
\email{tfylam@umich.edu}\thanks{T.L. was supported by NSF grants DMS-0901111 and DMS-1160726, and by a
Sloan Fellowship.  A.T. was supported by ARC grant DP110100440 and by an Australian Postdoctoral Fellowship.}

\author{Anne Thomas}\address{School of Mathematics and Statistics F07, University of Sydney NSW 2006, Australia}
\curraddr{School of Mathematics and Statistics, University of Glasgow, 15 University Gardens, Glasgow G12 8QW, United Kingdom}
\email{anne.thomas@glasgow.ac.uk}

\date{\today}

\begin{document}
\maketitle

\begin{abstract}  Let $(W,S)$ be a finite rank Coxeter system with $W$ infinite.  We prove that the limit weak order on the blocks of infinite reduced words of $W$ is encoded by the topology of the Tits boundary $\bTX$ of the Davis complex $X$ of $W$.  We consider many special cases, including $W$ word hyperbolic, and $X$ with isolated flats.   We establish that when $W$ is word hyperbolic, the limit weak order is the disjoint union of weak orders of finite Coxeter groups.   We also establish, for each boundary point $\xi$, a natural order-preserving correspondence between infinite reduced words which ``point towards" $\xi$, and elements of the reflection subgroup of $W$ which fixes $\xi$.
\end{abstract}

\section{Introduction}

Let $(W,S)$ be a finite rank Coxeter system with $W$ infinite.  In this paper we compare the {\it limit weak order} on the {\it infinite reduced words} of $W$ with the topology of the Tits boundary $\bTX$ of the Davis complex $X$ of $W$.

\subsection{Infinite reduced words and limit weak order}

Let $S = \{s_i \mid i \in I\}$ denote the simple generators of $W$, and $\{\alpha_i \mid i \in I\}$ denote the simple roots.  An infinite reduced word $\ii = i_1i_2i_3 \cdots$ is an infinite word with letters in $I$ such that each initial finite subword $i_1i_2 \cdots i_k$ is a reduced word for $W$.  The inversion set $\Inv(\ii)$ of an infinite reduced word $\ii$ is the set of positive roots
$$
\Inv(\ii) = \{s_{i_1}s_{i_2} \cdots s_{i_{k-1}} \alpha_{i_k} \mid k = 1,2,\ldots \}.
$$

In Section~\ref{s:infinite reduced words} we review the notion of a braid limit $\ii \to \jj$ of two infinite reduced words, which roughly says that $\jj$ can be obtained from $\ii$ by an infinite sequence of braid moves.  Braid limits give rise to an equivalence relation on infinite reduced words, and a partial order on the equivalence classes, called the limit weak order, and denoted $\limW$.  Equivalently, we have $\ii \leq \jj$ in the limit weak order if and only if $\Inv(\ii) \subseteq \Inv(\jj)$.  The limit weak order generalises the usual weak order of a Coxeter group (see for example \cite{Bjorner-Brenti}) to infinite reduced words.  

The poset $\limW$ is in general quite complicated, and can, for example, contain infinite intervals.  Following \cite{Lam-Pylyavskyy}, we decompose $\limW$ into {\it blocks}, allowing us to study $\limW$ by studying two more manageable problems.  We say that two infinite reduced words $\ii$ and $\jj$ are {\it in the same block} if $\Inv(\ii)$ and $\Inv(\jj)$ differ by finitely many elements.   The limit weak order descends to a partial order on blocks.

As we shall see, on the one hand, the limit weak order on blocks can be given a purely topological description in terms of the Tits boundary of the Davis complex.   On the other hand, the limit weak order restricted to a block $B(\ii)$ is the usual weak order of a (usually infinite) Coxeter group.  

Let us remark on the history of infinite reduced words and the limit weak order.  Cellini and Papi \cite{Cellini-Papi} and later Ito \cite{Ito} studied certain subsets of positive roots of an affine Weyl group $W$ called biconvex sets.  These biconvex sets contain as a special case the inversion sets of infinite reduced words.  In particular, the inversion sets of infinite reduced words of affine Weyl groups are completely classified in these works.  Cellini and Papi were motivated by the study of initial sections of Dyer's total reflection orders, while Ito was motivated by the study of convex bases of quantized universal enveloping algebras $U_q({\mathfrak g})$.  Lam and Pylyavskyy \cite{Lam-Pylyavskyy} introduced the limit weak order on infinite reduced words in the context of the theory of total positivity of loop groups.  They described this partial order for affine Weyl groups (explained later in the introduction).   Dyer \cite{dyer_weak} studied a generalisation of weak order to biconvex sets of an arbitrary Coxeter group, and formulated a number of conjectures on the lattice structure of this generalisation. 

Our present work may have applications in each of these contexts, for example to the yet-to-be-developed theory of total positivity of Kac--Moody groups, or, as pointed out to us by an anonymous referee, to the conjectures in \cite{dyer_weak}.

\subsection{Davis complex and its Tits boundary}
The Davis complex $X$ is a proper, complete, $\CAT(0)$ metric space on which the Coxeter group $W$ acts properly discontinuously and cocompactly by isometries \cite{davis-book}.   The visual boundary $\partial X$, consisting of equivalence classes of geodesic rays in $X$, can be equipped with the Tits metric, giving a metric space $\bTX$ which we call the Tits boundary of $X$ (see \cite{bridson-haefliger}).  We use the Davis complex $X$ rather than the classical Tits cone because $X$ has a cocompact $W$-action and its collection of walls is locally finite, and because the $\CAT(0)$ structure on the Davis complex allows us to define the Tits boundary and to use many results from $\CAT(0)$ geometry.   The Davis complex was introduced in order to study geometric and topological properties of infinite non-affine Coxeter groups.  The Tits boundary of a $\CAT(0)$ space generalises the Tits boundary of a Euclidean building or symmetric space of non-compact type, which is a spherical building.  We recall background on the Davis complex and its Tits boundary in Section~\ref{s:preliminaries}.

There is a bijection between reflections $r \in W$ and {\it walls} $M(r) \subset X$.  Each infinite reduced word $\ii$ naturally gives rise to a path $\gamma: [0,\infty) \to X$, starting in the identity chamber of $X$, so that the walls crossed by $\gamma$ can be identified with the inversion set $\Inv(\ii)$ of $\ii$ (see also Section~\ref{s:infinite reduced words}).  

The (possibly empty) boundary of a wall $M$ is a subset $\bM$ of the boundary $\bTX$.  We show in Section~\ref{s:walls} that these boundaries of walls induce an arrangement in $\bTX$ with many properties similar to a hyperplane arrangement; in particular, each boundary $\bM$ separates $\bTX$ into ``half-spaces'' $\bM^+$ and $\bM^-$.  We define an equivalence relation on $\bTX$ essentially via the facial structure of this arrangement.  For $\xi \in \bTX$, we let $C(\xi)$ denote the equivalence class of $\xi$.   An important technical tool in Section~\ref{s:walls} is the Parallel Wall Theorem of Brink and Howlett \cite{brink-howlett}.

\subsection{Limit weak order on blocks and the Tits boundary}

In Section~\ref{s:bTXi}, we associate to each infinite reduced word $\ii$ a nonempty subset $\bTXi \subset \bTX$ as follows.  The set $\bTXi$ consists of points $\xi \in \bTX$ such that there is a geodesic ray $[p,\xi) \subset X$ pointing in the direction of $\xi$, with the property that $\Inv(\ii)$ is the disjoint union of the walls crossed by $[p,\xi)$, and the walls separating the identity from $p$.  Our first main theorem is the following:

\begin{theorem}\label{thm:main}
The subsets $\bTXi \subset \bTX$ have the following properties:
\begin{enumerate}
\item
For each $\ii$ and $\jj$ we have $\bTXi = \bTX(\jj)$ or $\bTXi \cap \bTX(\jj) = \emptyset$.  We have $\bTXi = \bTX(\jj)$ if and only if $\ii$ and $\jj$ are in the same block.
\item Each $\bTXi$ is an equivalence class $C(\xi)$, and each equivalence class $C(\xi)$  is of the form $\bTXi$ for some infinite reduced word $\ii$.  Thus the $\bTXi$ form a partition of $\bTX$.
\item
Each $\bTXi$ is a path-connected, totally geodesic subset of $\bTX$.
\item
The closure of $\bTXi$ in $\bTX$ is the following union:
$$
\overline{\bTXi} = \bigcup_{\jj \leq \ii} \bTXj
$$
where $\leq$ denotes the limit weak order.
\end{enumerate}
\end{theorem}

Thus {\it the topology of the Tits boundary encodes the limit weak order on the blocks of infinite reduced words}.  We prove parts (1) and (2) of Theorem \ref{thm:main} in Section~\ref{s:bTXi}, using results from Section~\ref{s:walls}, and parts (3) and (4) of Theorem~\ref{thm:main} are established in Sections~\ref{s:proof 3} and~\ref{s:proof 4} respectively.

A natural direction to explore would be the question: {\it to what extent does the limit weak order on blocks (or the limit weak order itself) determine the homotopy or homeomorphism type of the Tits boundary?}

\medskip

We thank Jean L\'ecureux for the observation that some of the
combinatorial objects we consider in Theorem~\ref{thm:main} appear in a different guise in his
work with Caprace \cite{Caprace-Lecureux}. There is a bijection between equivalence
classes of infinite reduced words and elements of the minimal
combinatorial compactification $\mathcal{C}_1(X)$, and blocks of
infinite reduced words are essentially the same as the spaces $X^\xi$
defined in Section~5.1 of \cite{Caprace-Lecureux}.

When $W$ is right-angled, the Davis complex $X$ is naturally a $\CAT(0)$ cube complex,
which coincides with the $\CAT(0)$ cube complex constructed for an arbitrary
Coxeter system by Niblo and Reeves (see Lemma 6 of \cite{niblo-reeves}).  At least in the case that
$W$ is right-angled, there seems to be a close relationship between 
blocks of infinite reduced words and the simplices in the simplicial boundary of $X$.  The
simplicial boundary of a $\CAT(0)$ cube complex was introduced by Hagen
in \cite{hagen}.

\subsection{Limit weak order restricted to a block}

In Section \ref{s:walls2} we associate to each $\xi \in \bTX$ the subgroup $\Wxi$ of $W$ generated by the set $\Rxi$ of reflections in walls which have $\xi$ in their boundary.  The subgroup $\Wxi$ is itself a Coxeter group.  We prove that the set of reflections in $\Wxi$ is equal to $\Rxi$, and then apply results of Deodhar \cite{Deodhar} and Dyer \cite{Dyer} to establish our second main result.

\begin{theorem}\label{thm:Wxi intro} Suppose $C(\xi)= \bTXi$.  Then the block $B(\ii)$ is in bijection with the set of elements of $\Wxi$, and the limit weak order on the block $B(\ii)$ corresponds naturally to the weak order on $\Wxi$. 
\end{theorem}

\noindent This is proved as Theorem~\ref{t:partial order Wxi}.

\subsection{Examples}
Section \ref{s:classes} discusses many special cases and examples.  We briefly describe the most important here.

In Section~\ref{s:affine} we consider the case that $(W,S)$ is an irreducible affine Coxeter group.  Let $W_\fin$ be the corresponding finite Weyl group.  Then $\bTX$ is isometric to a sphere with dimension equal to the rank of $W$, and the partition of $\bTX$ by the $C(\xi)$ or $\bTXi$ is essentially the (spherical) braid arrangement of $W_\fin$.  Furthermore, for each piece $C = C(\xi)$ the set $\{\ii \mid \bTXi = C\}$ of equivalence classes of infinite reduced words is either a singleton, or can be identified with a possibly reducible affine Coxeter group of lower rank.  These statements follow from the works of Cellini--Papi \cite{Cellini-Papi}, Ito \cite{Ito} and Lam--Pylyavskyy \cite{Lam-Pylyavskyy}.

Another interesting case is when $W$ is word hyperbolic, as discussed in Section~\ref{s:hyperbolic}.  Such Coxeter groups were characterised by Moussong (see \cite{davis-book}).  We can characterise word hyperbolic Coxeter groups in several new ways, and the restriction of Theorem \ref{thm:main} to this case is particularly elegant:
\begin{theorem}\label{thm:hyp}
The following are equivalent: 
\begin{enumerate}
\item The group $W$ is word hyperbolic.
\item For any infinite reduced word $\ii$, the set $\bTXi$ consists of a single point $\xi = \xi(\ii)$.
\item For any $\xi \in \bTX$, the group $W(\xi)$ is finite.
\end{enumerate}
Moreover if $W$ is word hyperbolic then two infinite reduced words $\ii$ and $\jj$ are comparable only if $\ii$ and $\jj$ are in the same block, and the limit weak order restricted to a block $B(\ii)$ is isomorphic to the weak order of the (possibly trivial) finite Coxeter group $W(\xi(\ii))$.
\end{theorem}
Thus the limit weak order of a word hyperbolic $W$ is the disjoint union of weak orders of finite Coxeter groups.    (Jean L\'ecureux points out that the finiteness of $W(\xi(\ii))$ in Theorem \ref{thm:hyp} can also be deduced from the proof of Theorem 5.13 in \cite{Caprace-Lecureux}.)   From the point of view of braid relations and infinite reduced words, these results seem far from obvious, as we illustrate in  Example \ref{eg:2d} below.


\begin{example}\label{eg:2d}  Consider the Coxeter group
$$ W = \langle s_1, s_2, s_3, s_4, s_5 \mid s_i^2 = (s_i s_{i+1})^2 = 1 \rangle.  $$
The generators $s_i$ may be viewed as reflections in the sides of a regular right-angled hyperbolic pentagon and the Davis complex $X$ is the induced tessellation of the hyperbolic plane by such pentagons, depicted in Figure~\ref{fig:pentagons}.  (This figure uses the Poincar\'e disk model of the hyperbolic plane.  All of the pentagons in the tessellation are isometric, although they appear to become arbitrarily small.) In this example, the Tits boundary $\bTX$ is the boundary of the hyperbolic plane, that is, a circle, and the distance between any two points in $\bTX$ is infinite.  Thus in particular, $\bTX$ has the discrete topology.  Each wall is a hyperbolic geodesic connecting two distinct boundary points.  The boundaries of walls are dense in $\bTX$, by minimality of the action of $W$ on the boundary.  Furthermore, each point of $\bTX$ is in the boundary of at most one wall.  For otherwise, we would have two disjoint walls that become arbitrarily close, which is impossible since each chamber of $X$ is a regular pentagon with fixed side length.

\begin{figure}
\begin{center}
\begin{overpic}[scale=0.5]{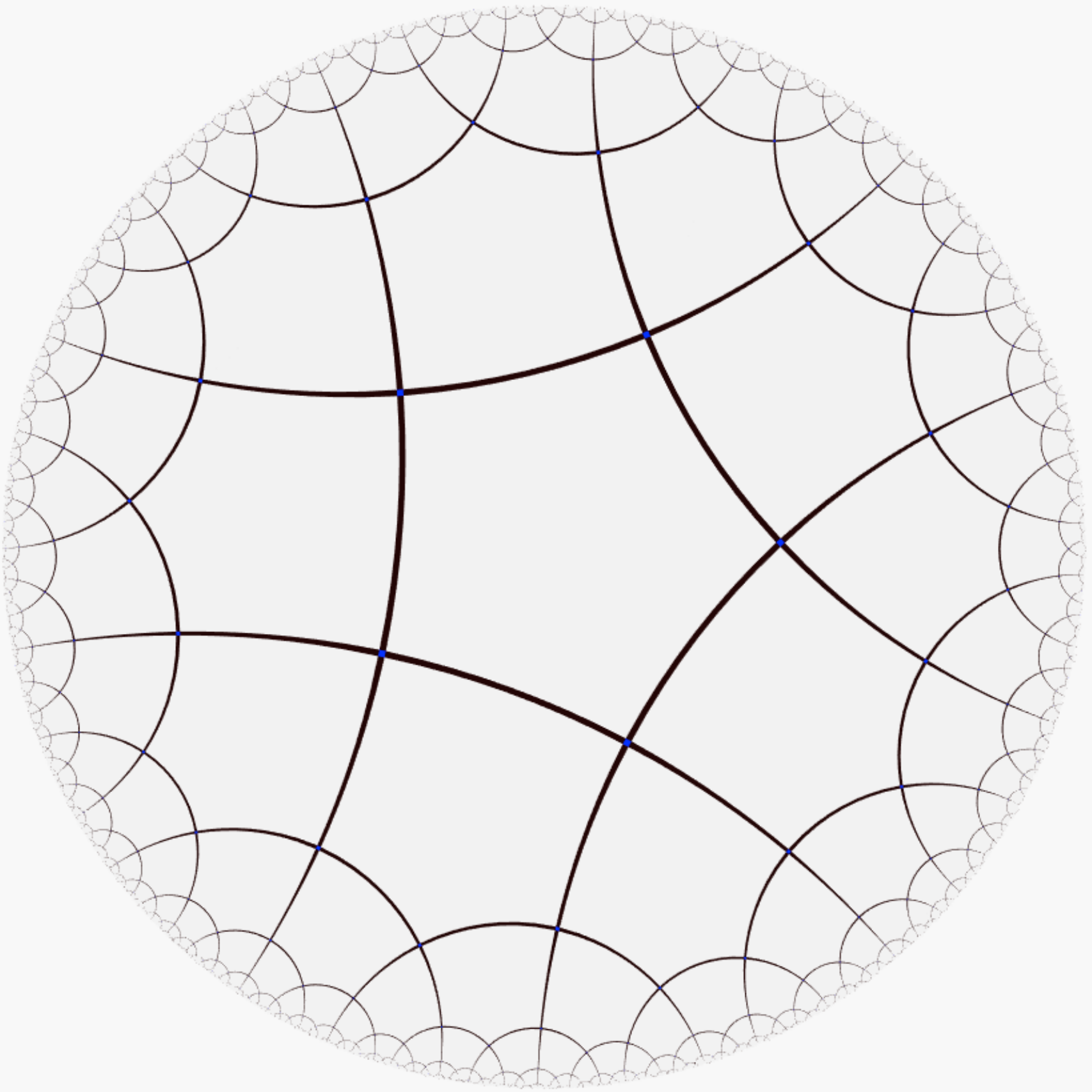}
\put(51,50){$e$}
\put(73,64){$1$}
\put(71,35){$2$}
\put(42,25){$3$}
\put(25,51){$4$}
\put(45,74){$5$}
\put(82,49){$12$}
\put(63,78){$15$}
\put(60,20){$23$}
\put(23,32){$34$}
\put(26,71){$45$}
\put(86,32){$25$}
\put(89,39){\small{125}}
\put(93.5,36){\tiny{1252}}
\put(92,32){\tiny{252}}
\put(66,59){$M$}
\put(97,33){$\xi$}
\put(86,65){$13$}
\put(80,76){$14$}
\put(86,73){\small{134}}
\put(91,70){\tiny{135}}
\put(90.5,72.5){\tiny{1354}}
\put(92.5,74){\tiny{13542}}
\end{overpic}
\end{center}
\caption{The Davis complex for Example \ref{eg:2d}}
\label{fig:pentagons}
\end{figure}
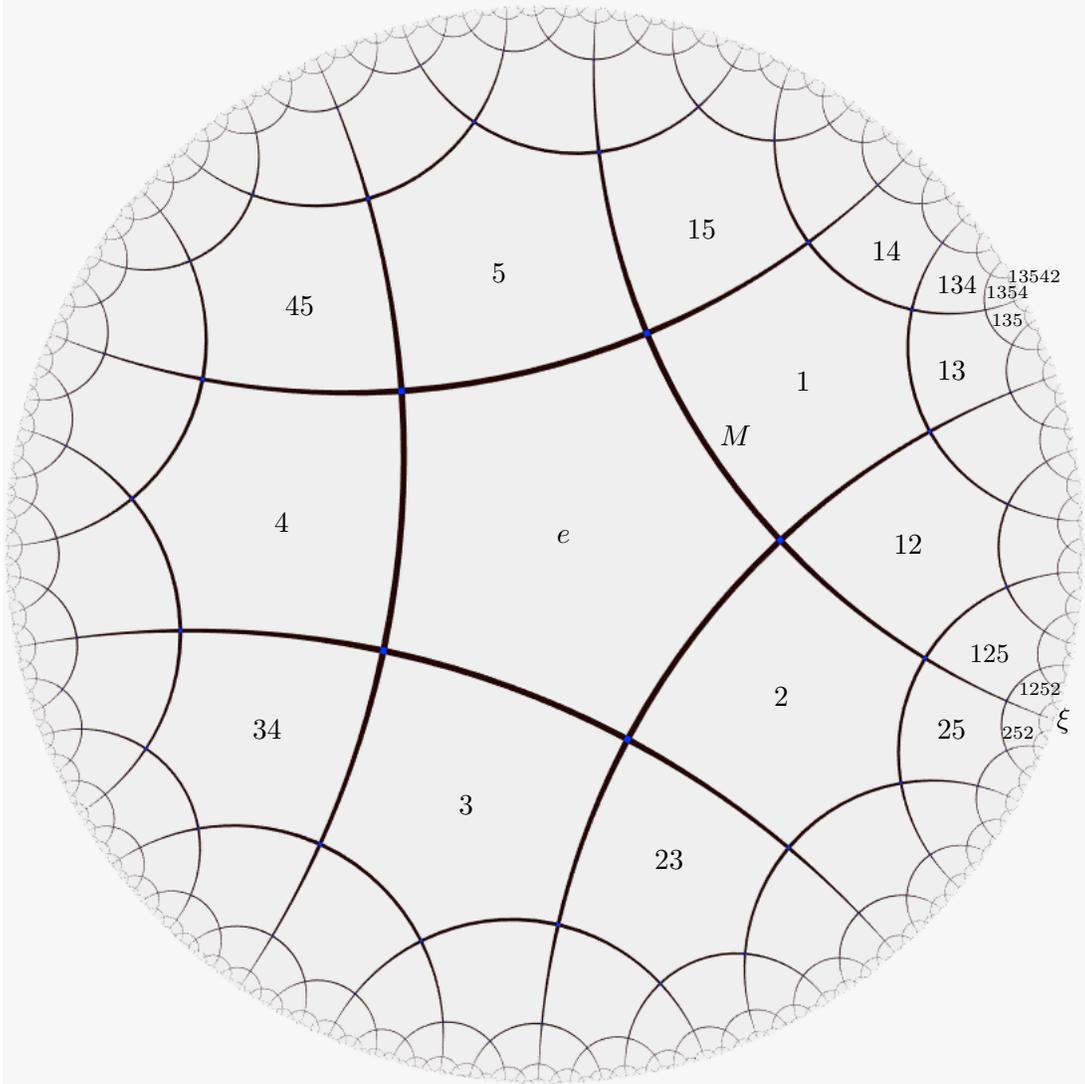

Suppose $\ii$ is an infinite reduced word with $\xi = \xi(\ii)$.  If $\xi$ is in the boundary of the wall $M = M(r)$ then the block $B(\ii)$ consists of two equivalence classes of infinite reduced words, corresponding to the two elements of the finite Coxeter group $W(\xi(\ii)) = \langle r \rangle$.  One of these equivalence classes consists of infinite reduced words which do not cross $M$ but which ``follow $M$ out to $\xi$", and this equivalence class corresponds to the trivial element of $W(\xi(\ii))$.  The other equivalence class in the block $B(\ii)$ consists of infinite reduced words which cross $M$ and then immediately ``follow $M$ out to $\xi$".  If $\xi$ is not in the boundary of any wall, then $B(\ii)$ contains a single equivalence class of infinite reduced words and $W(\xi(\ii))$ is trivial.

For example, as shown in Figure~\ref{eg:2d}, we have a braid limit
$$
\ii = 125252525\cdots  \rightarrow 25252525\cdots = \ii'
$$
of two inequivalent infinite reduced words $\ii$ and $\ii'$ such that $\xi = \xi(\ii) = \xi(\ii')$ lies on the boundary of the wall $M = M(s_1)$ (see Section \ref{s:infinite reduced words} for notation).  Thus $B(\ii) = B(\ii') = \{\ii,\ii'\}$ contains exactly two equivalence classes.  On the other hand the infinite reduced word
$$
\jj = 1352413524 \cdots 
$$
is such that $\xi(\jj)$ is not in the boundary of any wall.  To see this it is enough to check that there does not exist an infinite reduced word $\jj'$, not braid equivalent to $\jj$, satisfying either $\jj \to \jj'$ or $\jj' \to \jj$.  The fact that $\jj \to \jj'$ is impossible is easy: each initial finite subword of $\jj$ is the {\it unique} reduced word of the corresponding element of $W$.  So suppose $\jj' = j'_1j'_2 \cdots \to \jj$.  Then the letter $1$ must occur in $\jj'$ in such a way that it can be moved to the front via the Coxeter relations of $W$.  Thus $\jj'$ starts like one of $1 \cdots$, $21 \cdots$ or $51 \cdots$, or starts with a string of alternating $2$s and $5$s followed by a $1$.  The latter three cases are impossible, since $s_5$ and $s_2$ do not commute, and $s_5$ and $s_3$ do not commute.  Thus $\jj'$ starts with $1 \cdots$.  Continuing this argument we see that in fact $\jj' = \jj = 1352413524 \cdots$.

More generally, one can deduce the following result: an infinite reduced word $\ii = i_1i_2\cdots$ satisfies $\xi(\ii) \in \partial M$ for some wall $M$ if and only if there exists $N>0$ such that $i_Ni_{N+1}i_{N+2}i_{N+3}\cdots = a(a+2)a(a+2)\cdots$ for some $a\in \Z/5\Z$.
\end{example}

Another case where the topology of the Tits boundary $\bTX$ is known reasonably explicitly is the case where $X$ has {\it isolated flats}, which we discuss in Section \ref{s:isolated}.  Hruska and Kleiner~\cite{hruska-kleiner} have shown that the Tits boundary $\bTX$ of such spaces is a disjoint union of isolated points and standard Euclidean spheres, while Caprace \cite{caprace_isolated} has classified Coxeter groups $W$ where $X$ has isolated flats.  Using Caprace's work, we show that the partition of each Euclidean sphere in $\bTX$ into equivalence classes $C(\xi)$ is the arrangement induced by an affine Coxeter group.

The examples we consider in Section~\ref{s:classes} suggest the following questions: {\it does the Tits boundary of a Coxeter group always have the homotopy type of a disjoint union of a wedge of spheres?} {\it When is the limit weak order (or limit weak order on blocks) a disjoint union of posets $P$ that are graded, or Eulerian, or Cohen--Macaulay, or shellable?}  See for example \cite{Bjorner-Brenti} for a discussion of these poset properties in the Coxeter setting.

Throughout this paper, we make use of the standard theory of Coxeter groups, as found in for instance~\cite{humphreys}, and results about $\CAT(0)$ spaces from \cite{bridson-haefliger}.

\section{Infinite reduced words and the limit weak order}\label{s:infinite reduced words}

This section has brief background on infinite reduced words in Coxeter groups, and the limit weak order.   

Let $(W,S)$ be a Coxeter system of finite rank, where $S = \{s_i \mid i \in I\}$.   Thus $S$ is finite and $W$ has presentation
\[ W = \langle \{ s_i \}_{i \in I} \mid (s_is_j)^{m_{ij}} = 1\rangle \]
where $m_{ii} = 1$ for all $i \in I$, and for all distinct $i, j \in I$, $m_{ij} = m_{ji}$ and $m_{ij} \in \{ 2,3,4,\ldots\} \cup \{\infty \}$.  We shall assume that $W$ is an infinite group. 

An \emph{infinite reduced word} in $W$ is a sequence
\[\ii = i_1 i_2 i_3  \cdots \]
where each $i_k \in I$, and for each $n \geq 1$, the finite initial subword $i_1i_2\cdots i_n$ is a reduced word, or equivalently, the product $s_{i_1} s_{i_2} \cdots s_{i_n}$ is a reduced expression.

Let $\Phi^+$ denote the set of positive roots of $W$.  For each infinite reduced word $\ii = i_1 i_2 i_3  \cdots $, we define an {\it inversion set} $\Inv(\ii) \subset \Phi^+$ by
$$
\Inv(\ii) = \{s_{i_1}s_{i_2} \cdots s_{i_{k-1}} \alpha_{i_k} \mid k = 1,2,\ldots \}
$$
where $\alpha_{i_k}$ denotes the simple root indexed by $i_k$.  The fact that all the roots in $\Inv(\ii)$ are positive is a consequence of the reducedness condition.

There are bijections between the set $\Phi^+$ of positive roots, the set $R$ of reflections $r$ of $W$ and the set $\mathcal{M}$ of walls $M = M(r)$ in the Davis complex of Section \ref{s:preliminaries}.  We shall often identity $\Inv(\ii)$ with a set of reflections, or a set of walls, without further comment.  Thus $\Inv(\ii)$ also denotes the set of walls crossed by $\ii$ in the Davis complex.

Let $\jj$ and $\ii$ be infinite reduced words for $W$.  We shall say that
$\jj$ is a {\it braid limit} of $\ii$ if it can be obtained from $\ii$
by a possibly infinite sequence of braid moves.  More precisely, we
require that one has $\ii = \jj_0, \jj_1, \jj_2, \ldots$ such that
$\lim_{k \to \infty} \jj_k = \jj$ and for each $k$, there exists $\ell$ such that the finite initial subword $(\jj_k)_1(\jj_k)_2(\jj_k)_3\cdots(\jj_k)_\ell$ of $\jj_k$ is a reduced word for the same element of $W$ as the finite initial subword $(\jj_{k+1})_1(\jj_{k+1})_2(\jj_{k+2})_3\cdots(\jj_{k+1})_\ell$ of $\jj_{k+1}$, and we have $(\jj_k)_r = (\jj_{k+1})_r$ for $r > \ell$.  Here, the limit $\lim_{k
\to \infty} \jj_k = \jj$ of words is taken coordinate-wise: $j_r =
\lim_{k \to \infty} (\jj_k)_r$.  We write $\ii \to \jj$ to mean that there
is a braid limit from $\ii$ to $\jj$.  We say that $\ii$ and $\jj$ are (limit) {\it braid equivalent} if $\ii \to \jj$ and $\jj \to \ii$, and write $[\ii]$ for the equivalence class of $\ii$.

It is well known, see for example \cite[Proposition 3.1.3]{Bjorner-Brenti}, that for two Coxeter group elements $v,w \in W$, we have $v \leq w$ in weak order if and only if their inversion sets satisfy $\Inv(v) \subseteq \Inv(v)$.  In an analogous fashion, there are two equivalent ways to define the limit weak order on the set of braid equivalence classes of infinite reduced words.  The relation between the two possible definitions is given by the following result established in \cite[Section 4]{Lam-Pylyavskyy}.  See also \cite{Ito}, which uses different language.

\begin{lemma}\label{l:LP}
Two infinite reduced words $\ii$ and $\jj$ are braid equivalent if and only if $\Inv(\ii) = \Inv(\jj)$.  We have $\Inv(\ii) \subset \Inv(\jj)$ if and only if there is a braid limit from $\jj$ to $\ii$.  
\end{lemma}

Define a partial preorder on infinite reduced words by $\ii \leq \jj$ if and only if $\jj \to \ii$.  Note that by Lemma \ref{l:LP}, we have $\ii \leq \jj$ if and only if $\Inv(\ii) \subset \Inv(\jj)$.  This preorder descends to a partial order on braid equivalence classes of infinite reduced words, called the {\it limit weak order} in \cite{Lam-Pylyavskyy}.  We let $(\limW,\leq)$ denote the set of equivalence classes of infinite reduced words in $W$, equipped with the limit weak order.   

As in \cite{Lam-Pylyavskyy}, we shall divide $\limW$ into {\it blocks}.  As explained in the introduction, this  allows us to study the rather complicated poset $\limW$ by studying two simpler classes of posets.  We say that two infinite reduced words $\ii$ and $\jj$ are {\it in the same block} if $\Inv(\ii)$ and $\Inv(\jj)$ differ by finitely many roots.  We write $B(\ii) \subset \limW$ to denote the set of equivalence classes of infinite reduced words in the same block as $\ii$.  The partial order $\leq$ on $\limW$ induces a partial order on blocks: $B(\ii) \leq B(\jj)$ if there exist $[\ii'] \in B(\ii)$ and $[\jj'] \in B(\jj)$ so that $\ii' \leq \jj'$.  It does not seem easy to give a simple criterion for two infinite reduced words $\ii$ and $\jj$ to lie in the same block, without discussing inversion sets.


\section{Preliminaries on geometric group theory}\label{s:preliminaries}

This section recalls brief background on notions from geometric group theory, from the references \cite{bridson-haefliger} and \cite{davis-book}.  In Section~\ref{s:CAT0} we review geodesics and the $\CAT(0)$ condition.  We then recall the construction and relevant properties of the Davis complex $X$, and discuss the ends of $W$, in Section~\ref{s:davis complex}, and discuss boundaries of $X$, including the Tits boundary $\bTX$, in Section~\ref{s:boundaries}.

\subsection{Geodesics and the $\CAT(0)$ condition}\label{s:CAT0}

We briefly recall definitions concerning geodesics and the $\CAT(0)$ condition, referring the reader to \cite{bridson-haefliger} for details.

Let $(X,d)$ be a metric space.  A \emph{geodesic segment} from $x \in X$ to $y \in X$ is a map $\gamma:[0,l] \to X$ such that $\gamma(0) = x$, $\gamma(l) = y$, and $d(\gamma(t), \gamma(t')) = |t - t'|$ for all $t,t'\in [0,l]$.  Similarly we define \emph{geodesic rays} $\gamma:[0,\infty) \to X$ and \emph{geodesic lines} $\gamma:(-\infty,\infty) \to X$.  The metric space $X$ is said to be {\it geodesic} if every pair of points $x,y \in X$ is connected by a geodesic segment.

Let $X$ be a metric space.  Let $\Delta$ be a geodesic triangle in $X$, and $\overline{\Delta}$ be a Euclidean comparison triangle, that is, a triangle in the Euclidean plane with the same side lengths as $\Delta$.  There is then a bijection $x \mapsto \bar x$ between points of $\Delta$ and points of $\overline{\Delta}$.  We say that $\Delta$ is $\CAT(0)$ if for any $x,y \in \Delta$ we have $d(x,y) \leq d(\bar x,\bar y)$, and that $X$ is $\CAT(0)$ if $X$ is a geodesic metric space all of whose geodesic triangles are $\CAT(0)$.  Similarly one defines $\CAT(1)$ spaces using comparison triangles on the unit sphere in $\R^3$.

Let $\gamma: [0,a] \to X$ and $\gamma':[0,a'] \to X$ be two geodesic segments with the same start point $\gamma(0) = p = \gamma'(0)$.  The {\it (Alexandrov) angle} between $\gamma$ and $\gamma'$ is defined as
$$
\angle_p(\gamma,\gamma'): = \limsup_{t,t' \to 0} \overline{\angle}_p(\gamma(t),\gamma(t'))
$$
where $\overline{\angle}_p(\gamma(t),\gamma(t'))$ denotes the angle at $p$ of a Euclidean triangle which has side lengths equal to the pairwise distances between $\{p,\gamma(t),\gamma(t')\}$.

\subsection{The Davis complex and ends}\label{s:davis complex}

Let $(W,S)$ be a Coxeter system.  In this section we briefly recall the construction and relevant properties of the Davis complex $X$ for $(W,S)$, following~\cite{davis-book}.  We also discuss the ends of $W$ and show that if $W$ is one-ended then $X$ has the geodesic extension property (defined below).

For each $T \subset S$, let $W_T$ be the \emph{special subgroup} of $W$ generated by the elements of $T$.   We are following the terminology of Davis~\cite{davis-book} here; the subgroups $W_T$ are often called \emph{(standard) parabolic subgroups} in the literature.  By convention, $W_\emptyset$ is the trivial group.  We say $T \subset S$ is \emph{spherical} if $W_T$ is finite.  Denote by $\cS$ the set of all spherical subsets of $S$, partially ordered by inclusion.  The poset $\cS_{> \emptyset}$ is an abstract simplicial complex, denoted $L$ and called the \emph{nerve} of $(W,S)$.  Thus the vertex set of $L$ is $S$, and a nonempty set $T$ of vertices spans a simplex $\s_T$ in $L$ if and only if $T$ is spherical.

We denote by $K$ the geometric realisation of the poset $\cS$.  Equivalently, $K$ is the cone on the barycentric subdivision of the nerve $L$.  Note that $K$ is compact, since it is the cone on a finite simplicial complex.  We call the cone point of $K$ its \emph{centre}.  For each $s \in S$ let $K_s$ be the union of the (closed) simplices in $K$ which contain the vertex $s$ but do not contain the centre.  In other words, $K_s$ is the closed star of the vertex $s$ in the barycentric subdivision of $L$.  Note that $K_s$ has nonempty intersection with $K_t$ if and only if $s$ and $t$ generate a finite subgroup of $W$.  For each $x \in K$, let \[S(x):= \{ s \in S \mid x \in K_s \}.\]

Now define an equivalence relation $\sim$ on the set $W \times K$ by $(w,x) \sim
(w',x')$ if and only if $x = x'$ and $w^{-1}w' \in W_{S(x)}$.  The \emph{Davis complex} $X = X(W,S)$ for $(W,S)$ is then the quotient space \[ X := (W \times K) / \sim. \]  The natural $W$--action on $W \times K$ descends to an action on $X$.

We identify $K$ with the subcomplex $(e,K)$ of $X$, where $e$ is the identity element of $W$.  Then $K$, as well as any of its translates by an element of $W$, is called a \emph{chamber} of $X$.  It is immediate that each chamber is compact, the set of chambers is in bijection with the set of elements of $W$ and $W$ acts transitively on the set of chambers.  We denote by $0$ the centre of the chamber $K = (e,K)$ in $X$.

If $(W,S)$ is irreducible affine then the Davis complex $X = X(W,S)$ is just the barycentric subdivision of the Coxeter complex for $(W,S)$.

The Davis complex $X$ may be equipped with a piecewise Euclidean metric so that it is a proper, complete $\CAT(0)$ metric space.  Then in particular, given any two points $x, y \in X$, there is a unique geodesic segment from $x$ to $y$, denoted $[x,y]$.  Since $X$ is a complete $\CAT(0)$ metric space, by Lemma II.5.8(2) of \cite{bridson-haefliger} we may define $X$ to have the \emph{geodesic extension property} if every geodesic segment can be extended to a geodesic line.  

In the special case that $W$ is generated by the set of reflections in codimension one faces of a compact convex hyperbolic polytope $P$, the Davis complex $X$ may instead be equipped with a piecewise hyperbolic metric, so that $X$ is isometric to the induced tessellation of hyperbolic space by copies of $P$.  This natural metrisation of $X$ is used in Example \ref{eg:2d} above.  Although we work with a piecewise Euclidean metric on $X$, our results hold for this piecewise hyperbolic metric as well.

By construction, the Coxeter group $W$ acts properly discontinuously and cocompactly by isometries on $X$.  Hence $W$ with its word metric is quasi-isometric to $X$.

Recall that for a finitely generated group $G$, the \emph{ends} of $G$ counts the number of path components of the Cayley graph of $G$ as larger and larger finite subgraphs are removed.  By a result of Hopf, a finitely generated group $G$ has $0$, $1$, $2$ or infinitely many ends.  The first statement in the following theorem is due to Hopf, while the latter three statements are due to Davis and are Theorems 8.7.2, 8.7.3 and 8.7.4 of~\cite{davis-book}, respectively. 

\begin{theorem}\label{T:davisends}
Let $W$ be a Coxeter group.
\begin{enumerate}
\item $W$ has $0$ ends if and only if $W$ is spherical (that is, finite).
\item
$W$ is one-ended if and only if, for each $T \in \cS$, the
punctured nerve $L - \s_T$ is connected.
\item
$W$ is two-ended if and only if $(W, S)$ decomposes as the direct product
$(W, S) = (W_0 \times W_1, S_0 \cup S_1)$
where $W_1$ is finite and $W_0$ is the infinite dihedral group.
\item
$W$ has infinitely many ends if and only if it is infinite, not as in (3) and there is at least one $T \in \cS$ so that the punctured nerve
$L - \s_T$ is disconnected.
\end{enumerate}
\end{theorem}

The following result will be needed only in Sections~\ref{s:geometry Mxi} and~\ref{s:hyperbolic}, however we include it here since its proof uses details of the construction of the Davis complex.

\begin{proposition}\label{p:geodesic extension}  If $W$ is one-ended then $X$ has the geodesic extension property.
\end{proposition}

\begin{proof}  By \cite[Proposition II.5.10]{bridson-haefliger}, it suffices to show that $X$ has no free faces.  Let $K$ be the standard chamber of $X$ and $L$ the nerve of $W$.  Then by construction of the Davis complex, $X$ has no free faces if and only if every free face of $K$ is contained in a mirror of $K$.  Since $K$ is the cone on the barycentric subdivision of $L$, there is a free face of $K$ not contained in a mirror if and only if there is an $s \in S$ such that the vertex $\{ s \}$ is contained in a unique maximal simplex of $L$ of dimension $\geq 1$.  (Such a free face will contain the edge between $\{ s \}$ and the cone point of $K$.)

Now suppose $s \in S$ is contained in a unique maximal simplex.  Then the link of $s$ in $L$ is this maximal simplex with $s$ removed, which is itself a simplex $\s_T \in L$.
Thus $L$ is either a simplex (in which case $W$ has $0$ ends), or $L-\s_T$ is
disconnected. Then by Theorem \ref{T:davisends}, $W$
cannot be one-ended.
\end{proof}

\subsection{Visual boundary and Tits boundary of the Davis complex}\label{s:boundaries}

We briefly recall definitions and some results concerning these boundaries.  See \cite{bridson-haefliger} for details.  

We denote by $\bX$ the visual boundary of the Davis complex $X$.  That is, $\bX$ is the set of equivalence classes of geodesic rays in $X$, where two rays $c,c':[0,\infty) \to X$ are defined to be \emph{equivalent} if there is a $0 < K < \infty$ so that for all $t \geq 0$, $d(c(t),c'(t)) < K$.  We say that a point $\xi \in \bX$ is \emph{represented} by a geodesic ray $c$ if $c$ is in the equivalence class $\xi$, and we may then write $c(\infty) = \xi$.  Since $X$ is complete and $\CAT(0)$, for any $\xi \in \bX$ and $p \in X$, there exists a unique geodesic ray from $p$ to $\xi$, denoted $[p,\xi)$.

We denote by $\bTX$ the Tits boundary of the Davis complex $X$, that is, the visual boundary $\partial X$ equipped with the length metric $d_T$ induced by the angular metric $\angle$.  By definition, the distance between points $\xi,\xi'\in\bX$ in the angular metric is $\angle(\xi,\xi') = \sup_{p \in X}\angle_p(\xi,\xi')$, where $\angle_p(\xi,\xi')$ is the (Alexandrov) angle at $p$ between the geodesic rays $[p,\xi)$ and $[p,\xi')$.

\begin{lemma}\label{l:dichotomy} Let $\xi$ and $\xi'$ be distinct points in $\bTX$.  Then at least one of the following holds:
\begin{enumerate}
\item $\xi$ and $\xi'$ are connected by a geodesic arc in $\bTX$;
\item $\xi$ and $\xi'$ are connected by a geodesic in $X$.
\end{enumerate}
Moreover, if $\xi$ and $\xi'$ are not connected by a geodesic in $X$ and $d_T(\xi,\xi') < \pi$, then there is a unique geodesic arc in $\bTX$ which connects $\xi$ and $\xi'$.
\end{lemma}

\begin{proof} Since $X$ is a proper $\CAT(0)$ space, Proposition II.9.21(1) and (2) of \cite{bridson-haefliger} say that if $\xi$ and $\xi'$ are not connected by a geodesic in $X$, they are connected by a geodesic arc in the boundary $\bTX$, of length $d_T(\xi,\xi') = \angle(\xi,\xi') \leq \pi$.  By Theorem 9.20 of \cite{bridson-haefliger}, since $X$ is a complete $\CAT(0)$ space the boundary $\bTX$ is a complete $\CAT(1)$ space, so $d_T(\xi,\xi') < \pi$ implies that this geodesic arc is unique.   \end{proof}

We note that it is possible for both (1) and (2) in Lemma \ref{l:dichotomy} to occur.  For example if $W$ is irreducible affine of rank $n$, then $\bTX$ is the $(n-1)$-dimensional sphere with its usual metric, and every pair of antipodal points in $\bTX$ will be connected by infinitely many geodesic arcs of length $\pi$ in $\bTX$ as well as by infinitely many (parallel) geodesics in $X$.

\section{The limit weak order and the topology of the Tits boundary}\label{s:main proof}

In this section we define the terms in and then prove our first main result, Theorem~\ref{thm:main}, which is stated in the introduction.  In Section~\ref{s:walls} we use boundaries of walls to define the equivalence class $C(\xi)$ of a point $\xi \in \bTX$, and show that the induced arrangement in $\bTX$ has properties similar to a hyperplane arrangement.  Then in Section~\ref{s:bTXi} we define the subsets $\bTXi \subset \bTX$ and use results from Section~\ref{s:walls} to prove parts (1) and (2) of Theorem~\ref{thm:main}.  In particular, the sets $C(\xi)$ and $\bTXi$ induce the same partition of $\bTX$.  We then investigate the metric and topological properties of this partition, establishing parts (3) and (4) of Theorem~\ref{thm:main} in Sections~\ref{s:proof 3} and~\ref{s:proof 4}, respectively.

\subsection{The equivalence class $C(\xi)$}\label{s:walls}

In this section we use boundaries of walls to define an equivalence relation $\sim$ on $\bTX$, with $C(\xi)$ denoting the equivalence class of  $\xi \in \bTX$, and we establish some properties of this partition.  The separation properties of walls and their boundaries will be crucial in this work, and so we include a careful discussion of these topics before defining the relation $\sim$ in Definition~\ref{d:Cxi}.  We then define various inversion sets and relate them to $\sim$ (see in particular Proposition~\ref{p:infinite}).  We will use many results from this section in proving our main results. 

A \emph{wall} in $X$ is by definition the fixed set of a reflection $r \in W$.  Each wall $M = \Fix(r)$ is a closed, convex subcomplex of $X$, which determines two closed half-spaces $M^+$ and $M^-$ of $X$ such that $M = M^+ \cap M^-$ and the reflection $r$ interchanges $M^+$ and $M^-$.  The half-space $M^+$ is by definition that containing the point $0$.  The collection of walls in $X$ is locally finite, and the maximum number of walls intersecting at any point of $X$ is bounded by the maximum size of a subset $T \subset S$ for which $W_T$ is finite.

We will say that two points $x, y \in X$ are on \emph{opposite sides} of a wall $M$ if $x \in M^+ \setminus M$ and $y \in M^- \setminus M$, or vice versa.  We then say that a wall $M$ \emph{separates} two points $x, y \in X$ if $x$ and $y$ are on opposite sides of $M$, so in particular neither $x$ nor $y$ is in $M$.  By the local structure of walls, if $M$ separates $x$ and $y$ then the geodesic segment $[x,y]$ intersects $M$ in a single point.  Given $x, y \in X$, we denote by $\Mxy$ the set of walls which separate $x$ and $y$.

We now consider the boundaries of walls.  We define $\cM$ to be the closure of $M$ in $\overline{X} = X \cup \bX$, and define $\bM = \cM \cap \bX = \cM \setminus M$.  Similarly we define $\cM^+,\bM^+$ and $\cM^-,\bM^-$.

\begin{lemma}\label{l:halfspace}  Let $M$ be a wall in $X$.  
\begin{enumerate}
\item The sets $\bM$ and $\bM^\pm$ are closed in $\bTX$ and $\bM^\pm \setminus \bM$ is open in $\bTX$.
\item 
The intersection of $\bM^+$ and $\bM^-$ is $\bM$.
\end{enumerate}
\end{lemma}
\begin{proof}
For (1), by \cite[Proposition 9.7(1)]{bridson-haefliger}, it is enough to establish the claims for $\bX$.  But then the claims follow from the definitions.

For (2), the inclusion $\bM \subset \bM^+ \cap \bM^-$ is clear.  Now suppose $\xi \in \bM^+ \cap \bM^-$.  Since $M^+$ and $M^-$ are totally geodesic sets, $\xi$ can be represented by geodesic rays $c, c'$ where $c$ (respectively $c'$) stays completely inside $M^+$ (respectively $M^-$).  Since $c$ and $c'$ are equivalent we have $d(c(t),c'(t)) < K < \infty$ so in particular both $c$ and $c'$ are a bounded distance from $M$.  It follows that $\xi \in \bM$.
\end{proof}
 
We say that two points $\xi,\xi' \in \bX$ are on \emph{opposite sides of $\bM$} if $\xi \in \bM^+ \setminus \bM$ and $\xi' \in \bM^- \setminus \bM$, or vice versa.

\begin{lemma}\label{l:crossing boundaries}  Let $M$ be a wall.  Suppose $\xi, \xi' \in \bX \setminus \bM$ are on opposite sides of $\bM$.  If we are in the situation of Lemma \ref{l:dichotomy}(1), then the geodesic arc in $\bTX$ connecting $\xi$ and $\xi'$ crosses $\bM$.  If we are in the situation of Lemma \ref{l:dichotomy}(2), then the geodesic in $X$ connecting $\xi$ and $\xi'$ crosses $M$.
\end{lemma}
\begin{proof}
In the case of Lemma \ref{l:dichotomy}(1), let $\gamma \subset \bTX$ be the geodesic arc.  Then $\gamma \cap \bM^+$ and $\gamma \cap \bM^-$ are both closed subsets of $\gamma$, and must thus intersect.  Hence $\gamma \cap \bM^- \cap \bM^+ \neq \emptyset$.  In the case of Lemma \ref{l:dichotomy}(2), if the geodesic $\gamma \subset X$ does not intersect $M$, then it must stay completely inside either $M^+$ or $M^-$.  But then $\xi,\xi'$ will both lie in $\bM^+$ or $\bM^-$, contradicting the assumption.
\end{proof}

If $\xi, \xi' \in \bX \setminus \bM$ are on opposite sides of $\bM$, then using Lemma \ref{l:crossing boundaries}, we will abuse terminology and say that $\xi$ and $\xi'$ are on \emph{opposite sides of $M$} and are \emph{separated by $M$}.  Two points $\xi, \xi' \in \bX \setminus \bM$ are defined to be on the \emph{same side of $M$} if they are not on opposite sides.  

We are now ready to define the relation $\sim$, which partitions the boundary $\bTX$ into equivalence classes $C(\xi)$.

\begin{definition}[The set $C(\xi)$]\label{d:Cxi}
Define an equivalence relation on $\bTX$ by $\xi \sim \xi'$ if and only if for each wall $M$, $\xi \in \bM^\pm \Leftrightarrow \xi' \in \bM^\pm$. Let $C(\xi) \subset \bTX$ denote the equivalence class of $\xi$.
\end{definition}

\noindent That is, $\xi \sim \xi'$ if and only if, for every wall $M$ such that $\xi, \xi' \in \bX \setminus \bM$, $\xi$ and $\xi'$ are on the same side of $M$.

To explore the properties of this partition, we first consider geodesic rays from various points in the Davis complex to various points in its boundary.  The following lemma will be used repeatedly.

\begin{lemma}\label{l:rays and walls} Let $M$ be a wall.
\begin{enumerate}
\item The geodesic ray from any $p \in M$ to $\xi \in \bM$ must be contained completely in $M$.
\item Let $c$ be a geodesic ray which intersects $M$ but is not contained in $M$.  Then $c(\infty)$ is not in $\bM$.
\item Let $\xi \in \bM$ and $p \in X \setminus M$.  Then the geodesic ray $[p,\xi)$ never intersects $M$.
\end{enumerate}
\end{lemma}
\begin{proof}  For (1), if $[p,\xi)$ is not contained in $M$ then reflecting this ray in $M$ gives another geodesic ray from $p$ to $\xi$, contradicting the uniqueness of the geodesic ray.  For (2), by local considerations there is a unique point of intersection of $c$ with $M$, at say $c(t_0) = p_0$.  The unique geodesic ray from $p_0$ to $\xi = c(\infty)$ is then given by $c(t)$ for $t \geq t_0$.  If $\xi \in \bM$ this contradicts (1).  For (3), if $[p,\xi)$ intersects $M$ this contradicts (2).  \end{proof}

\noindent We thus make the following definition: a wall $M$ \emph{separates} $p \in X$ from $\xi \in \bX$ if $p$ is not in $M$ and the geodesic ray $[p,\xi)$ intersects $M$, hence by Lemma \ref{l:rays and walls}(3) the boundary point $\xi$ is not in $\bM$.

We now define several inversion sets for points in $X$ or its boundary $\bTX$, and relate them to the equivalence relation $\sim$.    First, for $\xi \in \bX$, let $\Inv(\xi)$ be the set of walls separating $\xi$ from $0$.  By our definition of separation, if $\xi$ lies in $\bM$, then $M \notin \Inv(\xi)$ regardless of the relative position of $0$ and $M$.  Next, define $\Inv(p)$ for $p \in X$ to be the set of walls separating $p$ from $0$.  Finally, define $\Inv(p,\xi)$ to be the set of walls separating $p$ from $\xi$, so that $\Inv(\xi) = \Inv(0,\xi)$.  A first observation, which we use in Section~\ref{s:bTXi} below, is the following.

\begin{lemma}\label{l:rays from p}  Suppose $\xi$ and $\xi'$ are equivalent.  Then for each $p \in X$, $\Inv(p,\xi) = \Inv(p,\xi')$. 
\end{lemma}

\begin{proof}  Suppose first that $M \in \Inv(p)$, that is, that $p \in M^- \setminus M$.  Then $\xi \in \bM^+ \setminus \bM$ (respectively, $\xi' \in \bM^+ \setminus \bM$) if and only if $[p,\xi)$ (respectively, $[p,\xi')$) crosses $M$, and $\xi \in \bM$ (respectively, $\xi' \in \bM$) if and only if $[p,\xi)$ (respectively, $[p,\xi')$) does not cross $M$.  The case that $p \in M^+ \setminus M$ is similar.  If $p \in M$ then neither $[p,\xi)$ nor $[p,\xi')$ crosses $M$.
\end{proof}

The next two lemmas will be used to prove Proposition~\ref{p:infinite} below, which says that if $\xi$ and $\xi'$ are not equivalent then $\Inv(\xi)$ and $\Inv(\xi')$ differ by finitely many walls (compare the definition of blocks in Section~\ref{s:infinite reduced words} above). Lemma~\ref{l:separated1} will also be used in Section~\ref{s:proof 4} below.  We shall need the following result \cite[Theorem 2.8]{brink-howlett}.

\begin{theorem}[Parallel Wall Theorem]  \label{thm:parallelwall} Let $(W,S)$ be a Coxeter system.  Then there is a constant $k = k(W,S)$ such that for every reflection $r \in R$ and every $x \in X$, if the distance in $X$ from $x$ to the wall $M(r)$ is greater than $k$, then there is another reflection $r' \in R$ such that $M(r')$ is disjoint from $M(r)$ and $M(r')$ separates $x$ from $M(r)$.
\end{theorem}

\begin{remark}\label{rem:parallel}  The statement of \cite[Theorem 2.8]{brink-howlett} is in the framework of root systems.  As remarked on p. 182 of \cite{brink-howlett} their Theorem 2.8 can be shown to be equivalent to a result called the Parallel Wall Theorem, stated as Theorem 1.7 but not completely proved in a preprint of Davis and Shapiro~\cite{davis-shapiro}.  The result \cite[Theorem 1.7]{davis-shapiro} applies to the Cayley graph of a Coxeter system $(W,S)$, rather than to the Davis complex.  Since the Cayley graph of $(W,S)$ is naturally quasi-isometric to the Davis complex, it is not hard to obtain from the statement of \cite[Theorem 1.7]{davis-shapiro} the formulation of Theorem \ref{thm:parallelwall} we give above.   
\end{remark}

\begin{lemma}\label{l:separated1} Suppose $\xi \in \bX$, $M$ is a wall and $p \in M$.  If $\xi \notin \bM$ then $\Inv(p,\xi)$ contains infinitely many pairwise disjoint walls $M_i$ that separate $\xi$ from $\bM$.
\end{lemma}

\begin{proof}  Let $c$ be the (unique) geodesic ray such that $c(0) = p \in M$ and $c(\infty) = \xi$.  By the Parallel Wall Theorem (Theorem~\ref{thm:parallelwall}), there is
a wall $M_1$ which separates $c(t_1)$ from $M$, for some sufficiently large
$t_1$.  Now, $c$ must intersect $M_1$ at some time $t'_1 \in (0,t_1)$, so in
particular $\xi$ does not lie in $\bM_1$, by Lemma \ref{l:rays and walls}(2).  It follows
that $c$ does not stay bounded distance from $M_1$, so we can repeat the
argument to find $M_2$, $M_3$ and so on.  
\end{proof}

\begin{lemma}\label{l:basepoint}  Let $\xi \in \bX$ and $p \in X$.  Then $$\left(\Inv(0,\xi) \setminus \Inv(p,\xi)\right) \cup \left(\Inv(p,\xi) \setminus \Inv(0,\xi)\right)$$ is finite.
\end{lemma}

\begin{proof}  Assume that there are infinitely many walls $M_i \in\Inv(0,\xi) \setminus \Inv(p,\xi)$.  Note that by definition $\xi \notin \bM_i$, so $\xi \in \bM_i^- \setminus \bM_i$ for each $i$.  Since only finitely many walls cross each point of the finite length geodesic segment $[0,p]$, and the intersection points are a discrete set, we may assume that $[0,p]$ does not cross any of the walls $M_i$.  Thus in particular $p \in M_i^+ \setminus M_i$ for each $i$.  But $\xi \notin \bM_i$ and $[p,\xi)$ does not cross $M_i$, so this implies $\xi \in \bM_i^+ \setminus \bM_i$ for each $i$, a contradiction.  The argument if there are infinitely many walls in $\Inv(p,\xi) \setminus \Inv(0,\xi)$ is similar.
\end{proof}

\begin{proposition}\label{p:infinite}
Suppose $\xi$ and $\xi'$ are not equivalent.  Then $$\left(\Inv(\xi)\setminus \Inv(\xi') \right) \cup \left(\Inv(\xi')\setminus \Inv(\xi)\right)$$ is infinite.
\end{proposition}

\begin{proof}
Without loss of generality, we may consider the following two cases: $\xi \in \bM^- \setminus \bM$ and $\xi' \in \bM^+$, and $\xi \in \bM^+ \setminus \bM$ and $\xi' \in \bM$, for some wall $M$.  In the first case, by Lemma \ref{l:separated1}, there are infinitely many walls which separate $\xi$ from $\bM$, and so there are infinitely many walls which separate $\xi$ from $\xi'$.  In the second case, let $p$ be a point in $M$.  By Lemma \ref{l:separated1} and Lemma \ref{l:basepoint}, there are infinitely many walls $M_i$ which are in both $\Inv(p,\xi)$ and $\Inv(\xi) = \Inv(0,\xi)$, and which separate $\bM$ from $\xi$.  None of these walls $M_i$ are in $\Inv(p,\xi')$, since $[p,\xi')$ is contained in $M$, so again by Lemma \ref{l:basepoint} only finitely many of the $M_i$ can be in $\Inv(\xi') = \Inv(0,\xi')$.  Hence there are infinitely many $M_i$ in $\Inv(\xi) \setminus \Inv(\xi')$.  This completes the proof. \end{proof}

\subsection{The sets $\bTXi$ and the proof of (1) and (2) of Theorem~\ref{thm:main}}\label{s:bTXi}

In this section we begin by defining the subsets $\bTXi \subset \bTX$ and establishing some initial properties of these sets.  The main result of this section is then Proposition~\ref{p:partition} below, which gives a new description, using the sets $\bTXi$, for the partition of $\bTX$ induced by the equivalence classes $C(\xi)$ (see Definition~\ref{d:Cxi} above).  Proposition~\ref{p:partition} is exactly the statement of  Theorem~\ref{thm:main}(2), and also implies the first statement in Theorem~\ref{thm:main}(1).  A corollary of Proposition~\ref{p:partition} will complete the proof of Theorem~\ref{thm:main}(1).

\begin{definition}[The set $\bTXi$]\label{d:bTXi}  Let $\ii$ be an infinite reduced word.  We define $\bTXi$ to be the set of $\xi \in \bTX$ such that there exists a  geodesic ray $c$ with $c(\infty) = \xi$ such that when $c(0) = p \in X$, the walls crossed by $c$, together with the walls separating $p$ from $0$, are exactly the inversions of $\ii$ (and no wall is crossed twice).
\end{definition}

\noindent Note that by Lemma \ref{l:LP}, if $\ii$ and $\jj$ are equivalent infinite reduced words, then $\bTXi = \bTXj$.  The following lemma provides some justification for Definition~\ref{d:bTXi}, by showing that the sets $\bTXi$ are nonempty and cover $\bTX$.

\begin{lemma}\label{l:defn}  
\begin{enumerate}
\item For each infinite reduced word $\ii$, the set $\bTXi$ is nonempty.
\item Every $\xi \in \bTX$ lies in some $\bTXi$.
\end{enumerate}
\end{lemma}

\begin{proof}  For the first part, let $0 = x_0,x_1, x_2, \ldots$ be the centres of the chambers visited by an infinite reduced word $\ii$.  Since the Davis complex $X$ is a proper metric space, the space $\overline{X} = X \cup \bX$ obtained by adjoining to $X$ its visual boundary is compact \cite[p. 264]{bridson-haefliger}.  Thus the sequence $\{ x_i \}$ has a convergent subsequence $\{ x_{i_n} \}$, with limit say $\xi \in \overline{X}$.  In fact, $\xi \in \bX$, since the $x_i$ are centres of chambers and so do not become arbitrarily close.  

As noted in Definition 8.5 of \cite{bridson-haefliger}, the sequence $\{ x_{i_n} \}$ in $X \subset \overline{X}$ converges to a point $\xi \in \bX$ if and only if the geodesic segments $[x_0,x_{i_n}]$ converge (uniformly on compact subsets) to the geodesic ray $c = [x_0,\xi)$.  We claim that this ray crosses the same set of walls as $\ii$, hence $\xi \in \bTXi$ and so $\bTXi$ is nonempty.  Suppose that a wall $M$ is in $\Inv(\ii)$.  Then there is an $n > 0$ such that $M$ separates $0$ from all of $x_{i_n}, x_{i_{n+1}}, \ldots$.  Thus for all large enough $n$ the geodesic segment $[x_0,x_{i_n}]$ crosses $M$, and so $c$ crosses $M$.  Now suppose $c$ crosses a wall $M$.  Then for all $n$ large enough, the geodesic segment $[x_0,x_{i_n}]$ crosses $M$, and thus the wall $M$ separates $x_0$ from $x_{i_n}$, hence $M \in \Inv(\ii)$.

For the second part, let $\xi \in \bTX$ and let $c:[0,\infty) \to X$ be a geodesic ray in the direction of $\xi$ such that $c(0)$ lies in the identity chamber.  The sequence of chambers encountered by $c$ will give rise to sequence $w^{(1)},w^{(2)},\ldots \in W$ of elements of the Coxeter group such that for each $i$, $\ell(w^{(i)}) = \ell(w^{(i-1)}) + \ell((w^{(i-1)})^{-1}w^{(i)})$.  In particular such a sequence arises from taking finite subsequences of a single infinite reduced word.  (Note that $\ell(w^{(i)}) - \ell(w^{(i-1)})>1$ can occur if multiple walls are crossed at one time.)  Finally, we remark that the sequence $w^{(1)},w^{(2)},\ldots \in W$ must be infinite by Lemma \ref{l:rays and walls}(2).
\end{proof}

We are now ready to state and prove the main result of this section, which is exactly the statement of Theorem~\ref{thm:main}(2).

\begin{proposition}\label{p:partition}
Each $\bTXi$ is an equivalence class $C(\xi)$, and each equivalence class $C(\xi)$ is of the form $\bTXi$ for some infinite reduced word $\ii$.  Thus the $\bTXi$ form a partition of $\bTX$.
\end{proposition}
\begin{proof}
We first show that each $\bTXi$ is a union of equivalence classes $C(\xi)$.  Suppose that $\xi$ and $\xi'$ are equivalent.  By Lemma~\ref{l:defn}(2), we have $\xi \in \bTXi$ for some infinite reduced word $\ii$.  Let $c$ be a geodesic ray  with $c(\infty) = \xi$ such that when $c(0) = p$, the walls crossed by $c$, together with the walls separating $p$ from $0$, are exactly the inversions of $\ii$ (with no wall crossed twice).  Now by Lemma~\ref{l:rays from p}, since $\xi$ and $\xi'$ are equivalent, the set of walls crossed by $[p,\xi)$ is the same as the set of walls crossed by $[p,\xi')$.  Thus by definition of $\bTXi$, we have $\xi' \in \bTXi$.  

We now show that each $\bTXi$ is equal to a single equivalence class $C(\xi)$.  Suppose now that $\xi$ and $\xi'$ are \emph{not} equivalent.  Then by Proposition~\ref{p:infinite}, there are infinitely many walls which separate one but not both of the two points $\xi$ and $\xi'$ from $0$.  For any $p \in X$, there are only finitely many walls separating $p$ from $0$.  It follows from the definition of $\bTXi$ that not both of $\xi$ and $\xi'$ can lie in $\bTXi$.  Thus each $\bTXi$ is equal to a single equivalence class $C(\xi)$.

To complete the proof, we note that by Lemma~\ref{l:defn}(2), the sets $\bTXi$ cover $\bTX$, and so as the $C(\xi)$ are equivalence classes the $\bTXi$ form a partition of $X$.
\end{proof}

Since the $\bTXi$ partition $\bTX$, we have also established the statement from Theorem~\ref{thm:main}(1) that for each $\ii$ and $\jj$, either $\bTXi = \bTXj$ or $\bTXi  \cap \bTXj = \emptyset$.  The remainder of Theorem~\ref{thm:main}(1) is a consequence of the definition of blocks and the next result.

\begin{corollary}\label{c:bTXi} We have $\bTXi = \bTXj$ if and only if $\Inv(\ii)\setminus \Inv(\jj)$ and $\Inv(\jj)\setminus \Inv(\ii)$ are both finite sets.
\end{corollary}
\begin{proof}  This follows from Proposition~\ref{p:partition} together with Lemma~\ref{l:basepoint} and Proposition~\ref{p:infinite}.
\end{proof}

We note another corollary of Proposition~\ref{p:partition}, which will be used in Section~\ref{s:proof 3}.

\begin{corollary}\label{c:same walls} Let $\ii$ be an infinite reduced word and suppose $\xi, \xi' \in \bTXi$.  Then for each $p \in X$, 
$\Inv(p,\xi) = \Inv(p,\xi')$.
\end{corollary}

\begin{proof}  This is immediate from Proposition~\ref{p:partition} together with Lemma~\ref{l:rays from p}.
\end{proof}

\subsection{Proof of (3) of Theorem~\ref{thm:main}}\label{s:proof 3}

We now move to describing the geometric and topological properties of our partition of $\bTX$ into sets $C(\xi)$ or $\bTXi$.  In this section we establish Theorem~\ref{thm:main}(3).  Since we have shown that the sets $C(\xi)$ and $\bTXi$ induce the same partition, Theorem~\ref{thm:main}(3) is equivalent to Proposition~\ref{prop:path-connected} below, which states that each $C(\xi)$ is a path-connected, totally geodesic subset of $\bTX$.

In order to prove Proposition~\ref{prop:path-connected}, the first result needed is the following.  This lemma will also be used in our proof of Theorem~\ref{thm:main}(4) in Section~\ref{s:proof 4} below.

\begin{lemma}\label{l:finite set hyperplanes} Let $\xi$ and $\xi'$ be distinct points in $\bTX$, with $\xi = c(\infty)$ and $\xi' = c'(\infty)$ where $c$ and $c'$ are geodesic rays based at the same point $c(0) = c'(0) = p \in X$.  If $\xi$ and $\xi'$ are connected by a geodesic $\gamma$ in $X$ then the set of walls crossed by both $c$ and $c'$ is finite.
\end{lemma}
\begin{proof}
Suppose $M$ is a wall such that $c$ and $c'$ both intersect $M$, and $M$ does not contain $c(0) = c'(0)$. Then $M$ separates $c(\infty)$ and $c'(\infty)$ from $c(0) = c'(0)$.  Since $\gamma$ is a geodesic it can intersect $M$ at most once, but $c(\infty)$ and $c'(\infty)$ are on the same side of $M$, so $\gamma$ does not intersect $M$ at all.  Thus $M$ separates $\gamma$ from $c(0)=c'(0)$.

Now let $\alpha$ be any geodesic segment connecting $c(0)=c'(0)$ to some point on $\gamma$.  Every wall $M$ which intersects both $c$ and $c'$ must intersect $\alpha$.  But only finitely walls go through each point, and the intersection points on $\alpha$ are a discrete set.  Since $\alpha$ has finite length the statement of the lemma follows.
\end{proof}

The proof of Proposition~\ref{prop:path-connected} will also use the following two technical results.

\begin{lemma}\label{l:wall intersection} For each $L > 0$ there is a constant $f(L)$ so that a geodesic segment of length $L$ intersects at most $f(L)$ walls.
\end{lemma}

\begin{proof}  A sufficiently small ball in $X$ intersects at most as many walls as the maximum number incident at a point.  A geodesic segment of length $L$ is covered by a fixed number (depending on $L$) of these sufficiently small balls, which completes the proof.
\end{proof}

The \emph{angle} between a geodesic $c$ and a wall $M$ that it meets is by definition the infimum of the angles between $c$ and geodesics contained in $M$.  

\begin{lemma}\label{l:angle}  There exist constants $L > 0$ and $\epsilon > 0$ such that any geodesic segment of length $L$ must intersect some wall at angle greater than $\epsilon$.
\end{lemma}

\begin{proof}  Suppose otherwise, that is, suppose that for all $L > 0$ and $\epsilon > 0$ there is a geodesic segment of length $L$ such that every wall it intersects is met at angle $\leq \epsilon/2$.  By Lemma \ref{l:wall intersection}, this geodesic segment intersects at most $f(L)$ walls.  By reflecting this geodesic segment, we may consider the images of its subsegments inside the fundamental chamber.  Let this ``broken geodesic" have initial point $a_1$ and final point $a_n$ and let $a_2,\ldots,a_{n-1}$ be the points at which the path turns.  In other words the maximal geodesic subsegments inside the fundamental chamber are $[a_k,a_{k+1}]$ for $1 \leq k < n$.  Note that $n \leq f(L)$, $\sum_{k = 1}^{n-1} d(a_k,a_{k+1}) = L$, and for $1 \leq k < n$ the angles between the geodesic segments $[a_k,a_{k+1}]$ and the walls containing the points $a_k$ and $a_{k+1}$ are both $\leq \epsilon/2$.  By \cite[I.1.13(2)]{bridson-haefliger} it follows that for $1 \leq k \leq n-2$ the angle between $[a_k,a_{k+1}]$ and $[a_{k+1}, a_{k+2}]$ is greater than $\pi - \epsilon$.  To simplify notation we will denote by $\angle(a_i a_j, a_ja_k)$ the angle between the geodesic segments $[a_i,a_j]$ and $[a_j, a_k]$, for $1 \leq i,j,k \leq n$.

We claim that for every $L > 0$, there is an $\epsilon = \epsilon(L)$ such that $d(a_0,a_n) > L/\sqrt{2}$.  This suffices to complete the proof, since for $L$ large enough there are no points inside the fundamental chamber at distance $L/\sqrt{2}$ apart.

We first show that for $1 \leq k < n$, the angle between $[a_0,a_k]$ and $[a_k,a_{k+1}]$ is greater than $\pi - k\epsilon$.  Suppose by induction that $\angle(a_0a_k,a_ka_{k+1}) > \pi - k\epsilon$.  In a $\CAT(0)$ space the angle sum of a triangle is at most $\pi$, so it follows that $\angle(a_0a_{k+1},a_{k+1}a_k) \leq k\epsilon$.  Let $a_0'$ be a point such that the geodesic segment $[a_0, a_0']$ contains $a_{k+1}$ in its interior and let $a_k'$ be a point such that the geodesic segment $[a_k, a_k']$ contains $a_{k+1}$ in its interior.  Since $\angle(a_ka_{k+1},a_{k+1}a_{k+2}) > \pi - \epsilon$, we have $\angle(a_k'a_{k+1},a_{k+1}a_{k+2}) \leq \epsilon$, and since $\angle(a_0a_{k+1},a_{k+1}a_k) \leq k\epsilon$ we also have $\angle(a_k'a_{k+1},a_{k+1}a_0') \leq k\epsilon$.   By \cite[Proposition I.1.14]{bridson-haefliger} it follows that $\angle(a_0'a_{k+1},a_{k+1}a_{k+2})$ is less than or equal to $(k+1)\epsilon$, so $\angle(a_0a_{k+1},a_{k+1}a_{k+2}) > \pi - (k+1)\epsilon$ as required.

In particular, $\angle(a_0a_{k},a_{k}a_{k+1}) > \pi - n\epsilon$ for each $k$.  Now using \cite[Proposition II.1.7(5)]{bridson-haefliger}, we get that for every $k$ we have
\begin{equation}\label{E:Cne}
d(a_0,a_{k+1}) > C(n\epsilon) \left(d(a_0,a_k) + d(a_k,a_{k+1}) \right)
\end{equation}
 for some constant $C$ depending only on $n\epsilon$.  The function $C = C(\delta)$ has the property that $C \to 1$ as $\delta \to 0$.  In fact, one can pick $C(\delta) = \sqrt{(1-\cos(\pi-\delta))/2}$. Repeatedly using \eqref{E:Cne}, we obtain
$$
d(a_0,a_n) > C(n\epsilon)^n \sum_{k=0}^{n-1} d(a_k,a_{k+1}) = C(n\epsilon)^n L.
$$
Now pick $\epsilon = \epsilon(L) > 0$ so that $1 > C(f(L)\,\epsilon) > (1/\sqrt{2})^{1/f(L)}$.  Then with this $\epsilon$, every geodesic segment of length $L$ must intersect some wall at angle greater than $\epsilon$.
\end{proof}

The last result needed for the proof of Proposition~\ref{prop:path-connected} is the following.

\begin{lemma}\label{l:distance} Suppose $\xi, \xi' \in \bTXi$.  Then $d_T(\xi,\xi') < \pi$. \end{lemma}

\begin{proof}  It follows from Lemma \ref{l:finite set hyperplanes} that if $\xi' \in C(\xi)$ then $\xi$ and $\xi'$ are not connected by a geodesic in $X$.  Hence by Proposition 9.21(1) and (2) of \cite{bridson-haefliger}, $d_T(\xi,\xi') = \angle(\xi,\xi') \leq \pi$.  Suppose that $d_T(\xi,\xi') = \pi$.  Fix $\varepsilon > 0$.  Then there is a point $p \in X$ such that $\angle_p(\xi,\xi') >  \pi - \varepsilon$.  Let $c$ be the geodesic ray from $p$ to $\xi$ and $c'$ the geodesic ray from $p$ to $\xi'$.  Then by Corollary \ref{c:same walls}, $c$ and $c'$ cross the same set of walls.  By Lemma \ref{l:angle}, there is a wall $M$ that intersects $c$ at angle greater than $\varepsilon$.  This wall $M$ also intersects $c'$, and the triangle with vertices the two intersection points and $p$ will have total angle greater than $\pi$.  This is a contradiction.
\end{proof}

\begin{proposition}\label{prop:path-connected}
Each $C(\xi)$ is a path-connected, totally geodesic subset of $\bTX$.
\end{proposition}
\begin{proof}
It follows from Lemma \ref{l:finite set hyperplanes} that if $\xi' \in C(\xi)$ then $\xi$ and $\xi'$ are not connected by a geodesic in $X$.  Hence by Lemma \ref{l:dichotomy} and Lemma \ref{l:distance}, there is a unique geodesic arc in $\bTX$ connecting $\xi$ and $\xi'$, of length strictly less than $\pi$.  We will show that each point $\eta$ on the geodesic arc in $\bTX$ connecting $\xi'$ and $\xi$ in fact lies in $\bTXi$.

By Proposition \ref{p:partition}, it suffices to show that $\eta \in C(\xi)$.  Suppose that there is a wall $M$ such that $\xi,\xi' \in \bM^+$ but $\eta \in \bM^- \setminus \bM$.  Then the arc $[\xi,\xi']$ has a (possibly equal) subarc $\gamma = [\zeta,\zeta']$ such that $\zeta, \zeta' \in \bM$ and every point in the interior of $\gamma$ is in $\bM^- \setminus \bM$.  Let $r$ be the reflection in the wall $M$.  Then $r$ induces an isometry of the Tits boundary, which fixes $\zeta$ and $\zeta'$ but does not fix $\gamma$.  But then $\gamma$ and $r(\gamma)$ are two distinct geodesic arcs in $\bTX$ connecting $\zeta$ and $\zeta'$ of length strictly less than $\pi$.  This is impossible, so there is no wall $M$ such that $\xi,\xi' \in \bM^+$ but $\eta \in \bM^- \setminus \bM$.  The argument is similar if $\xi,\xi' \in \bM^-$ but $\eta \in \bM^+ \setminus \bM$.  Therefore $\eta \in C(\xi)$, and so $C(\xi)$ is a path-connected totally geodesic subset of $\bTX$. \end{proof}

This completes the proof of (3) of Theorem~\ref{thm:main}.

\subsection{Proof of (4) of Theorem~\ref{thm:main}}\label{s:proof 4}

We now complete the proof of Theorem~\ref{thm:main} by establishing (4), which describes the closure of the sets $\bTXi$.  Since each $\bTXi$ is an equivalence class $C(\xi)$, we first consider the closures of the $C(\xi)$, in Proposition~\ref{p:closure}.   Theorem~\ref{t:closure partition} then proves Theorem~\ref{thm:main}(4).

We will need the following definitions.  For each $M$, let
$$
\epsilon_M(\xi) = \begin{cases} 0 & \mbox{if $\xi \in \bM$} \\
+ & \mbox{if $\xi \in \bM^+ \setminus \bM$} \\
- & \mbox{if $\xi \in \bM^- \setminus \bM$.}
\end{cases}
$$
Define $\cbM^\pm := \bM^\pm \setminus \bM$ and $\cbM^0 := \bM$.  Then
$$
C(\xi) = \bigcap_{M} \cbM^{\epsilon_M(\xi)}.
$$
Now let 
$$
C'(\xi) = \bigcap_{M} \bM^{\epsilon_M(\xi)}
$$
where $\bM^0 :=\bM$.

\begin{proposition}\label{p:closure}
The closure $\overline{C(\xi)}$ is equal to $C'(\xi)$.
\end{proposition}
\begin{proof}
Let $\eta' \in C'(\xi)$ and $\eta \in C(\xi)$.  Suppose $\eta'$ and $\eta$ are in the situation of Lemma \ref{l:dichotomy}(1).  Then as in the proof of Proposition \ref{p:infinite}, there are infinitely many walls $M$ that separate $\eta'$ from $0$ which do not separate $\eta$ from $0$.  This is impossible from the definition of $C'(\xi)$.  Thus there is a geodesic arc $\gamma:[0,l] \to \bTX$ with $\gamma(0) = \eta'$ and $\gamma(l) = \eta$.  The interior of this arc cannot intersect any walls, again by the definition of $C'(\xi)$.  Thus $\gamma((0,l]) \subset C(\xi)$, and so $\gamma(0) \in \overline{\gamma((0,l])} \subset \overline{C(\xi)}$.  Finally, we note that $C'(\xi)$ is closed, since it is an intersection of closed subspaces (Lemma \ref{l:halfspace}(1)).
\end{proof}

The following result proves Theorem~\ref{thm:main}(4), and so completes the proof of Theorem \ref{thm:main}.

\begin{theorem}\label{t:closure partition}
The $\bTXi$ form a partition of $\bTX$ satisfying
$$
\overline{\bTXi} = \bigsqcup \{ \bTXj \mid \jj \mbox{ a set of representatives for the blocks $\leq B(\ii)$}\} = \bigcup_{\jj \leq \ii} \bTXj.
$$
\end{theorem}

\begin{remark}
Note that it is not true that if $\bTX(\jj') \subset \overline{\bTXi}$ then $\jj' \leq \ii$.  One may have to first increase $\Inv(\ii)$ by a finite set of walls before this inequality holds.
\end{remark}

\begin{proof}
Let $\bTXi = C(\xi)$ and $\xi' \in \bTX$.  Suppose $C(\xi') \subset \overline{C(\xi)}$.  It follows from Proposition \ref{p:closure} that $\Inv(\xi') \subset \Inv(\xi)$.  Let $p \in X$ be such that the set of walls crossed from $0$ to $p$, and then from $p$ to $\xi$, is equal to $\Inv(\ii)$ (with no wall crossed twice).  Then using Lemma \ref{l:rays and walls}(3), one sees that the set of walls crossed from $0$ to $p$, and then from $p$ to $\xi'$, includes no wall twice, and is equal to $\Inv(p) \cup \Inv(\xi')$.  Thus $\overline{\bTXi} \subset \bigcup_{\jj \leq \ii} \bTXj$.

Conversely, suppose $\jj$ is such that $C(\xi') = \bTXj \notin \overline{C(\xi)}$.  If $\xi'$ is not in the same path-component of $\bTX$ as $C(\xi)$, then by Lemma \ref{l:finite set hyperplanes}, we know that $\jj$ and $\ii$ are incomparable.  Now assume that $\xi$ and $\xi'$ are in the same path component of $\bTX$.  Let $\gamma:[0,l] \to \bTX$ be the geodesic arc such that $\gamma(0) = \xi'$ and $\gamma(l) = \xi$.  If the interior of $\gamma$ intersects some wall $\bM$, then by Lemma \ref{l:separated1} there is a line of pairwise non-intersecting walls $M_i$ separating $\xi$ from $\bM$, and hence separating $\xi$ from $\xi'$.  Note that the origin (or any other point) cannot lie on the same side of all these walls.  Otherwise any (finite length) geodesic joining the origin to $M$ will intersect infinitely many of these walls.  It follows that $\Inv(\xi) \setminus \Inv(\xi')$ and $\Inv(\xi') \setminus \Inv(\xi)$  are both infinite, and so $\jj$ and $\ii$ are incomparable.

Finally, we consider the case that $\gamma((0,l))$ intersects no walls, but that $\xi$ is contained in some boundary of a wall $\bM$, and $\xi'$ is not contained in $\bM$.  In this case, the same argument shows that $\Inv(\xi')\setminus \Inv(\xi)$ is infinite and so we cannot have $\jj < \ii$.
\end{proof}

\section{The reflection group $W(\xi)$}\label{s:walls2}

In this section we study the group $W(\xi)$ generated by the reflections in the set $\Mxi$ of walls which have $\xi \in \bTX$ in their boundary.  We prove Theorem~\ref{thm:Wxi intro} of the introduction in Section~\ref{s:proof Wxi}, and then in Section~\ref{s:geometry Mxi} relate the geometry of $X$ to $\Mxi$.  The results of Section~\ref{s:geometry Mxi} will be used in several of the special cases we consider in Section~\ref{s:classes} below.

\subsection{Proof of Theorem~\ref{thm:Wxi intro}}\label{s:proof Wxi}

We now define the group $W(\xi)$ and prove Theorem~\ref{thm:Wxi intro}, which is stated in the introduction.  

Given $\xi \in \bTX$ define $\Mxi$ to be the (possibly empty) set of walls in $X$ which have $\xi$ in their boundary, that is,
\[ \Mxi := \{ M \in \mathcal{M} \mid \xi \in \bM \}.\]
Now for each $\xi \in \bTX$, define $\Rxi$ to be the (possibly empty) set of reflections in walls in $\Mxi$.  That is, $\Rxi$ is the set of reflections in walls which have $\xi$ in their boundary.  If $\Rxi$ is nonempty, let $\Wxi$ be the subgroup of $W$ generated by the elements of $\Rxi$.  If $\Rxi$ is empty, then define $\Wxi$ to be the trivial group.

By a theorem of Deodhar \cite{Deodhar}, proved independently by Dyer \cite{Dyer}, the group $\Wxi$ is itself a Coxeter group, with respect to a canonical system of generators.  Dyer \cite[Corollary (3.4)(i)]{Dyer} also established the following.

\begin{lemma}\label{L:Dyer}  Let $w \in \Wxi$ be arbitrary and $r \in \Wxi$ be a reflection.  Then $\ell(rw) < \ell(w)$ in $\Wxi$ if and only if $\ell(rw) < \ell(w)$ in $W$.
\end{lemma}

In order to say more about $\Wxi$, we consider the set $\Rxi$ and the set of reflections of $\Wxi$.

\begin{lemma}\label{l:Rxi}  Let $\xi \in \bTX$.
\begin{enumerate}
\item The set $\Rxi$ is closed under conjugation, that is, if $r,t$ lie in $\Rxi$ then $rtr^{-1}$ lies in $\Rxi$.
\item Every reflection in $\Wxi$ is conjugate via an element of $\Wxi$ to a reflection in $\Rxi$.
\end{enumerate}
\end{lemma}

\begin{proof}  For the first part, we have to show that if we reflect $M(t)$ in $M(r)$ then the resulting wall $M' = r.M(t)$, which is fixed by the reflection $rtr^{-1}$, lies in $\Mxi$.  Let $c$ be a geodesic contained in $M(t)$ such that $c(\infty) = \xi$.  Let $c' = r.c$ be its reflection, which is contained in $M'$.  Since $\xi$ is fixed by the reflection $r$, we have $c'(\infty) = \xi = c(\infty)$, and so $M' \in \Mxi$. 

The second part is Corollary (3.11)(ii) in Dyer \cite{Dyer}.
\end{proof}

An immediate consequence of Lemma~\ref{l:Rxi} is that:

\begin{corollary}\label{c:reflections in Wxi}  The set $\Rxi$ is exactly the set of reflections of $\Wxi$.  
\end{corollary}

For use in Section~\ref{s:hyperbolic} we note:

\begin{corollary}\label{c:Rxi finite}  The group $\Wxi$ is finite if and only if the set of walls $\Mxi$ is finite.
\end{corollary}

\begin{proof}  This follows from Corollary~\ref{c:reflections in Wxi} and the observation that $\Wxi$ is finite if and only if it contains finitely many reflections.
\end{proof}

For $w \in W$ and $\xi \in \bTX$, denote by $\Inv_{\Mxi}(w)$ the set $\Inv(w) \cap \Mxi$.  The next lemma is used in the proof of Theorem~\ref{t:partial order Wxi}.

\begin{lemma}\label{L:InvWxi}
Suppose $w \in W$.  Then there exists $w' \in \Wxi$ such that $\Inv_{\Mxi}(w)$ can naturally be identified with $\Inv_{\Wxi}(w')$.
\end{lemma}
\begin{proof}
 Suppose first that $w$ lies in  $\Wxi$.  Then it follows from Lemma \ref{L:Dyer} that $\Inv_{\Mxi}(w)$ can naturally be identified with the inversion set $\Inv_{\Wxi}(w)$, when $w$ is considered as an element of $\Wxi$.  So take $w' = w$.

Now suppose that $w$ is not in $W(\xi)$.  
We proceed by induction on $k = |\Inv_{\Mxi}(w)|$.  For the base case $k = 0$, take $w'$ trivial.  For $k >0$, pick a wall $M \in \Inv_{\Mxi}(w)$ so that no other walls in $\Mxi$ separate $w$ from $M$.  Let the reflection corresponding to $M$ be denoted $r \in \Wxi$.  Let $v \in W$ be such that $\Inv_{\Mxi}(v) = \Inv_{\Mxi}(w) \setminus \{M\}$, and by induction we suppose that we have found $v' \in \Wxi$ such that $\Inv_{\Mxi}(v)=\Inv_{\Wxi}(v')$.  We claim that $w' = rv'$ works, where geometrically $w'$ is obtained from $v'$ by reflecting in $M$.  This follows easily from the observation that no wall in $\Mxi$ separates $v'$ from $M$, which in turn follows from the same property of $v$.
\end{proof}

We now prove Theorem~\ref{t:partial order Wxi}, which establishes Theorem~\ref{thm:Wxi intro} in the introduction.  

\begin{theorem}\label{t:partial order Wxi}  Suppose $\Wxi$ is nontrivial.  Then the set of equivalence classes of infinite reduced words $\ii$ such that $C(\xi)= \bTXi$ is in bijection with the set of elements of $\Wxi$, and the partial order on this set of equivalence classes of infinite reduced words corresponds naturally to the weak partial order on $\Wxi$.
\end{theorem}

\begin{proof}  Let $w$ be an element of $\Wxi$.  The following lemma and its corollaries will allow us to construct an infinite reduced word $\ii$ such that $C(\xi) = \bTXi$ and $\Inv(\ii) = \Inv_{\Mxi}(w) \sqcup \Inv(\xi)$.

\begin{lemma}\label{l:p for Wxi}  Let $p_0$ be in the interior of the chamber of $X$ corresponding to $w$.
Then there is a point $p$ on the geodesic ray joining $p_0$ to $\xi$, such that every wall intersecting the geodesic segment $[0,p]$ either has $\xi$ in its boundary, or separates $0$ from $\xi$.
\end{lemma}

\begin{proof}  Let $M$ be a wall intersecting the geodesic segment $[0,p_0]$.  If $M$ does not intersect the geodesic ray $[p_0,\xi)$ and $M$ does not separate $0$ from $\xi$, then since neither of the geodesic rays $[p_0,\xi)$ and $[0,\xi)$ cross $M$, and $M$ separates $0$ from $p_0$, we must have $\xi \in \bM$, that is, $M \in \Mxi$.

Now suppose the geodesic segment $[0,p_0]$ intersects a wall $M_1$ such that $M_1 \not \in \Mxi$ and $M_1$ does not separate $0$ from $\xi$.   Then by the above argument, $M_1$ intersects $[p_0,\xi)$.  So we may define $p_1$ to be a point on the geodesic ray
 $[p_0, \xi)$ such that $p_1$ is on the same side of $M_1$ as $\xi$.  Note that the geodesic segment $[0,p_1]$ intersects fewer walls that do not have $\xi$ in their boundary and do not separate $0$ from $\xi$ than does the geodesic segment $[0,p_0]$.  So repeating the
construction, we build $p_2$, $p_3$,\dots, until we obtain $p = p_k$ with the
desired property.  (Note that the initial geodesic segment from $0$ to $p_0$
intersects only finitely many walls, so this process stops.)
\end{proof} 

The following corollaries and their proofs continue notation from Lemma~\ref{l:p for Wxi}.

\begin{corollary}  Let $M$ be a wall crossed by $[0,p]$ such that $\xi \in \bM$.  Then $M \in \Inv_{\Mxi}(w)$.  Thus $\Inv(p)$ is the disjoint union of $\Inv_{\Mxi}(w)$ with a subset of $\Inv(\xi)$.
\end{corollary}

\begin{proof}  Since $M \in \Mxi$ it suffices to show that $M \in \Inv(w)$, that is, that $M$ is crossed by $[0,p_0]$.  But the geodesic ray $[p_0,\xi)$ does not cross $M$, so $p_0$ and $p$ are on the same side of $M$.  Since $M$ separates $0$ from $p$ the result follows.
\end{proof}

\begin{corollary}  The set of walls crossed by $[0,p]$, together with the set of walls crossed by $[p,\xi)$, is the disjoint union of $\Inv_{\Mxi}(w)$ with $\Inv(\xi)$ (and no wall is crossed twice).
\end{corollary}

\begin{proof}  By construction, no wall crosses both $[0,p]$ and $[p,\xi)$.  If $M \in \Inv(\xi)$ and $M$ is not crossed by $[0,p]$, then $M$ separates both $0$ and $p$ from $\xi$, and so $M$ is crossed by $[p,\xi)$.
\end{proof}

\noindent Continuing with the proof of Theorem~\ref{t:partial order Wxi}, consider the sequence of chambers visited by $[0,p] \cup [p,\xi)$, where $p$ is constructed as in Lemma~\ref{l:p for Wxi}.  Then similarly to the proof of Lemma~\ref{l:defn}(2), this sequence of chambers gives rise to an infinite reduced word $\ii$ such that $\xi \in \bTXi $, hence $C(\xi) = \bTXi$, and $\Inv(\ii) = \Inv_{\Mxi}(w) \sqcup \Inv(\xi)$.

Conversely, if $\ii$ is an infinite reduced word such that $\xi \in \bTXi$, let $p$ be a point such that the set of walls crossed by $[0,p] \cup [p,\xi)$ is equal to $\Inv(\ii)$ (and no wall is crossed twice).  Without loss of generality we may choose $p$ to be in the interior of a chamber of $X$.  Let $p$ belong to the interior of the chamber for $w' \in W$, so that $\Inv(p) = \Inv(w')$.  Using Lemma \ref{L:InvWxi}, we define $w \in \Wxi$ to be the unique element such that $\Inv_{\Mxi}(w) = \Inv(w') \cap \Mxi$.  We claim that $\Inv(\ii)$ is the disjoint union of $\Inv_{\Mxi}(w)$ with $\Inv(\xi)$.  To see this, let $M \in \Inv(\ii)$.  If $M$ intersects $[p,\xi)$ then $M$ separates $0$ from $\xi$ so $M \in \Inv(\xi)$.  If $M$ does not intersect $[p,\xi)$ then $M$ intersects $[0,p]$, so either $\xi \in \bM$ or $M \in \Inv(\xi)$.  If $M$ intersects $[0,p]$ and $\xi \in \bM$ then $M \in \Inv_{\Mxi}(w)$ by definition.  Therefore $\Inv(\ii) \subset \Inv_{\Mxi}(w) \sqcup \Inv(\xi)$.   The opposite inclusion is clear.  Thus to each $\ii$ such that $\bTXi = C(\xi)$, we have associated an element $w \in \Wxi$ such that $\Inv(\ii) = \Inv_{\Mxi}(w) \sqcup \Inv(\xi)$.

These constructions are inverses, and the natural correspondence between the partial order on equivalence classes of infinite reduced words in $\bTXi$ and the weak partial order on $\Wxi$ follows from these constructions.
\end{proof}

\subsection{The geometry of $X$ and the set $\Mxi$}\label{s:geometry Mxi}

In this section we investigate the set $\Mxi$ and its implications for the geometry of $X$.  The main results are Propositions~\ref{p:flat} and~\ref{p:constant distance}, which both give sufficient conditions for $X$ to contain an isometrically embedded Euclidean plane.  Many results from this section will be used in Section~\ref{s:classes} below.

We first note that:

\begin{lemma}\label{l:pairwise}  Any collection of pairwise intersecting walls in $X$ is finite.
\end{lemma}

\begin{proof}  This follows from Lemma 3 of Niblo--Reeves \cite{niblo-reeves}.
\end{proof}

\begin{corollary}\label{c:Mxi infinite}  For each $\xi \in \bTX$, the set $\Mxi$ is infinite if and only if there are disjoint walls $M, M' \in \Mxi$.
\end{corollary}

\begin{proof}  Since the collection of walls is locally finite in $X$, if $\Mxi$ is infinite the conclusion follows from Lemma \ref{l:pairwise}.  Conversely, suppose $M, M' \in \Mxi$ are disjoint and let $r, r'$ be the reflections fixing $M, M'$ respectively.  Then $r$ and $r'$ generate an infinite dihedral group which fixes $\xi$.  Apply this group to the wall $M$ to obtain infinitely many walls in $\Mxi$.
\end{proof}

\begin{proposition}\label{p:flat}  Suppose that $M, M'$ are disjoint walls in $\Mxi$.  Then $X$ contains an isometrically embedded Euclidean plane.
\end{proposition}

\begin{proof}
Since $M$ and $M'$ are closed subsets of the complete metric space $X$, $M$ and $M'$ are complete in the induced metric.  Hence by \cite[Proposition II.2.4(1)]{bridson-haefliger}, we may choose points $p \in M$ and $p'\in M'$ such that $d(p,p') = d(M,M')$.  Let $c=[p,\xi)$ and $c'=[p',\xi)$ and consider the function  $t \mapsto d(c(t), c'(t))$.  Since the geodesic rays $c$ and $c'$ are equivalent and $d(p,p')$ realises the distance $d(M,M')$, this function is bounded and non-decreasing.  As $X$ is $\CAT(0)$, this function is convex.  Hence the function $t \mapsto d(c(t), c'(t))$ is constant.

Now for any $t > 0$, consider the four points $p$, $p'$, $q = c(t)$ and $q' = c'(t)$.  Since the wall $M'$ is a convex and complete subset of $X$, and $d(p,p')$ realises the distance $d(p,M')$, by \cite[Proposition II.2.4]{bridson-haefliger} the Alexandrov angle between the geodesic segments $[p,p']$ and $[p',q']$ is $\geq \pi/2$.  By considering $d(p',M)$ instead, we obtain that the Alexandrov angle between the geodesic segments $[p',p]$ and $[p,q]$ is also $\geq \pi/2$.  But since $d(q,q') = d(p,p')=d(M,M')$, the same argument shows that the Alexandrov angles between $[q,q']$ and the segments $[q,p]$ and $[q',p']$ are also $\geq \pi/2$.  Thus by the Flat Quadrilateral Theorem \cite[II.2.11]{bridson-haefliger}, each of these four angles is equal to $\pi/2$, and moreover the convex hull of the four points $p,p',q,q'$ is isometric to the convex hull of a rectangle in the Euclidean plane.  Therefore the convex hull of $c$ and $c'$ is a flat half-strip which is orthogonal to both $M$ and $M'$.

Let $r$ and $r'$ be the reflections fixing $M$ and $M'$ respectively.  Since $M$ and $M'$ are disjoint these reflections generate an infinite dihedral group.  The union of the images of the flat half-strip bounded by $c$ and $c'$ under the action of $\langle r, r'\rangle$ is a flat half-plane in $X$ (which is orthogonal to all of the images of $M$ and $M'$ under this action).  Hence for all $n \geq 1$ there is an isometric embedding into $X$ of the ball of radius $n$ centred at the origin in the Euclidean plane.  Thus by \cite[Lemma 9.34]{bridson-haefliger}, there is an isometric embedding of the Euclidean plane into $X$.
\end{proof}

\begin{corollary}\label{c:infinitexi}
If $\Mxi$ is infinite then $X$ contains an isometrically embedded Euclidean plane.
\end{corollary}

Recall that by Proposition~\ref{p:geodesic extension}, $X$ has the geodesic extension property if, for instance, $W$ is one-ended.  The next two lemmas and also Proposition~\ref{p:constant distance} will be used in Section~\ref{s:hyperbolic}, when we consider the case that $W$ is word hyperbolic, although they do not require hyperbolicity.

\begin{lemma}\label{l:Mxi intersection}  If $\Mxi$ is finite and nonempty then $\bigcap_{M \in \Mxi} M$ is nonempty.  If in addition $X$ has the geodesic extension property, then this intersection contains a geodesic line.
\end{lemma}

\begin{proof}  Since $\Wxi$ is finite and $X$ is $\CAT(0)$, there is a point $p \in X$ which is fixed by all $w \in \Wxi$.  In particular, $p$ is fixed by all reflections in $\Rxi$, hence $p$ is contained in every wall $M \in \Mxi$.  That is, $\cap_{M \in \Mxi} M$ is nonempty.  Now consider the geodesic ray $[p,\xi)$.  By Lemma \ref{l:rays and walls}, this geodesic ray is contained in every $M \in \Mxi$.  Let $\eta$ be the other endpoint of a geodesic extension of the ray $[p,\xi)$.  Then by the local properties of walls, the ray $[p,\eta)$ is also contained in every $M \in \Mxi$, so the intersection of the walls in $\Mxi$ is nonempty.  If $X$ has the geodesic extension property then by extending $[p,\xi)$ we obtain that $\cap_{M \in \Mxi} M$ contains a geodesic line.\end{proof}

\begin{lemma}\label{l:distinct bM}  Assume $X$ has the geodesic extension property.  If $M$ and $M'$ are distinct walls such that $M \cap M' \neq \emptyset$, then $\bM \neq \bM'$.
\end{lemma}

\begin{proof}  Let $p \in M \cap M'$.  Since $M$ and $M'$ are distinct walls, we may choose $q \in M \setminus M'$.  Then as $X$ has the geodesic extension property, the geodesic segment $[p,q]$ extends to a geodesic ray $[p,\xi)$.  By the local properties of walls, $\xi \in \bM$.  If $\xi \in \bM'$ as well, then by Lemma \ref{l:rays and walls} the entire geodesic ray $[p,\xi)$ is contained in $M'$, a contradiction.
\end{proof}

\begin{proposition}\label{p:constant distance}  Assume $X$ has the geodesic extension property.  Let $M$ and $M'$ be disjoint walls such that $\bM = \bM'$ is nonempty.  Then $M$ and $M'$ are constant distance apart and $X$ contains an isometrically embedded Euclidean plane which is orthogonal to both $M$ and $M'$.   \end{proposition}

\begin{proof}  Let $\xi \in \bM = \bM'$.  As in the proof of Proposition \ref{p:flat}, we may choose points $p \in M$ and $p' \in M'$ such that $d(p,p') = d(M,M')$, and the geodesic rays $[p,\xi)$ and $[p',\xi)$ then bound a flat half-strip in $X$.

Now since $X$ has the geodesic extension property, there exists a geodesic line $c:\R \to X$ such that $c([0,\infty)) = [p,\xi)$.  By the local structure of walls, the entire image of $c$ must be contained in $M$.   Put $\eta = c(-\infty)$.  Then $\eta \in \bM = \bM'$.  Since $d(p,p') = d(M, M')$, the geodesic rays $[p,\eta)$ and $[p',\eta)$ also bound a flat half-strip.

Let $r'$ be the reflection which fixes the wall $M'$.  Then $r'c$ is a geodesic line which is contained in the wall $r'M$ and passes through the point $r'p$.  The geodesic ray $[r'p,\xi)$ is constant distance from $[p',\xi)$, and the geodesic ray $[r'p,\eta)$ is constant distance from $[p',\eta)$.  Thus the geodesic line $r'c$ is at uniformly bounded distance from the geodesic line $c$.  Therefore by the Flat Strip Theorem \cite[II.2.13]{bridson-haefliger}, the convex hull of $c$ and $r'c$ is isometric to a flat strip.  Note that the width of this flat strip is $d(p,r'p) = 2d(p,p')$.

We have so far shown that there is a geodesic line $c$ in $M$, and passing through $p$, which is at constant distance from the geodesic line $r'c$ in $r'M$.

Now let $x$ be any point in $M$ which is distinct from $p$.  Then by the geodesic extension property again, the geodesic segment $[p,x]$ may be extended to a geodesic line $c_x$ passing through $p$, and moreover by the local structure of walls the line $c_x$ is contained in $M$.  The above argument may be repeated with the endpoints of $c_x$ to show that $c_x$ is at constant distance $2d(p,p')$ from the geodesic line $r'c_x$ in $r'M$.  In particular, $x$ is at distance $2d(p,p')$ from $r'M$.  Therefore we have that the walls $M$ and $r'M$ are at constant distance $2d(p,p')$.

Now for any point $x' \in M'$, we have $d(x',M) = d(x',r'M)$.  Thus the wall $M'$ is exactly halfway in between $M$ and $r'M$, and we conclude that $M$ and $M'$ are at constant distance $d(p,p')$.

To obtain an isometrically embedded Euclidean plane which is orthogonal to both $M$ and $M'$, keep on reflecting the flat strip between lines $c$ and $r'c$.
\end{proof}

\section{Special cases}\label{s:classes}

We conclude by discussing numerous special cases and examples.  These include the cases that $W$ is irreducible affine (Section~\ref{s:affine}), $W$ is reducible or virtually abelian (Section~\ref{s:reducible}), $W$ has two ends (Section~\ref{s:2 ends}), $W$ has infinitely many ends (Section~\ref{s:infinite ends}) and $W$ is word hyperbolic (Section~\ref{s:hyperbolic}), and the case that $X$ has isolated flats (Section~\ref{s:isolated}).  In particular, in Section~\ref{s:hyperbolic} we prove Theorem~\ref{thm:hyp} of the introduction.

\subsection{Irreducible affine Coxeter groups}\label{s:affine}

Infinite reduced words for affine Weyl groups were studied by Cellini--Papi \cite{Cellini-Papi} and Ito \cite{Ito} from a root system point of view.  The limit weak partial order on equivalence classes of infinite reduced words was studied in \cite{Lam-Pylyavskyy}.  We summarise the results here in the notation of the present paper.

Let $(W,S)$ be an irreducible affine Coxeter group with corresponding finite Weyl group $W_\fin$.  Then the boundaries of walls $\partial M \subset \bTX$ are in bijection with the reflections of $W_\fin$.  The pieces of the partition $\bTX = \sqcup C(\xi)$ are in bijection with the faces of the Coxeter arrangement of $W_\fin$ (with the origin of the Coxeter arrangement excluded).  This arrangement could also be described as the (thin) spherical building at infinity.

The pieces $C(\xi)$ which are open (and thus not contained in the closure of any other $C(\xi')$) are in bijection with $w \in W_\fin$.  For any other piece $C(\xi)$, we have that $\Mxi$ is infinite but is partitioned into finitely many parallelism classes.  These parallelism classes correspond to a collection of walls in the Coxeter arrangement of $W_\fin$.  Furthermore, $W(\xi)$ is a (possibly reducible) affine Coxeter group.

Thus the maximal elements of $(\limW, \leq)$ are in bijection with $W_\fin$.  The blocks of $\limW$ are in bijection with the faces of the Coxeter arrangement of $W_\fin$ with the origin excluded.  The limit weak order on each block $B(\ii)$ is isomorphic to the weak order on some affine Coxeter group.

\subsection{Reducible Coxeter groups and virtually abelian Coxeter groups}\label{s:reducible}

Suppose that $(W,S)$ is a reducible Coxeter system, with $S = S_1 \sqcup S_2$ so that $W = W_1 \times W_2$ where $W_k = \langle S_k \rangle$ for $k = 1,2$.  Let $I_k = \{ i \mid s_i \in S_k \}$.  Assume first that $W_1$ and $W_2$ are both infinite and let $\ii$ be an infinite reduced word in $W$.  Then exactly one of the following occurs:
\begin{enumerate}
\item $\ii$ is equivalent to an infinite reduced word consisting of a finite reduced word in $W_2$ followed by an infinite reduced word in $W_1$;
\item $\ii$ is equivalent to an infinite reduced word consisting of a finite reduced word in $W_1$ followed by an infinite reduced word in $W_2$; or
\item $\ii$ contains infinitely many elements of $I_1$ and infinitely many elements of $I_2$.
\end{enumerate}

If case (1) holds for two infinite reduced words $\ii$ and $\ii'$, denote by $w_2$ and $w_2'$ the elements of $W_2$  and by $\jj$ and $\jj'$ the infinite reduced words in $W_1$ such that $\Inv(\ii) = \Inv(w_2) \sqcup \Inv(\jj)$ and $\Inv(\ii') = \Inv(w_2') \sqcup \Inv(\jj')$.  Then $\ii \leq \ii'$ if and only if both $w_2 \leq w_2'$ in the weak partial order on $W_2$ and $\jj \leq \jj'$ in the partial order on infinite reduced words in $W_1$, with strict inequality $\ii < \ii'$ if and only if, in addition, at least one of the containments $\Inv(w_2) \subseteq \Inv(w_2')$ or $\Inv(\jj) \subseteq \Inv(\jj')$ is strict.  Similarly for the partial order on pairs of infinite reduced words for which (2) holds.  If (1) holds for $\ii$ and (2) holds for $\ii'$, then $\ii$ and $\ii'$ are clearly incomparable.

If case (3) holds then for $k = 1,2$, the infinite reduced word $\ii$ is strictly greater in the partial order than any infinite reduced word $\jj_k$ obtained by concatenating a finite initial subword of $\ii$ with the infinite reduced word consisting of all subsequent elements of $I_k$ in $\ii$.  Note that in this case the set $\Inv(\ii)$ differs from $\Inv(\jj_k)$ by infinitely many walls.

Now let us discuss the Tits boundary $\bTX$ of the Davis complex $X$ of $W$.  The Davis complex $X$ is a product $X = X_1 \times X_2$, and thus the Tits boundary is the spherical join $\bTX = \bTX_1 * \bTX_2$ (see \cite[Corollary II.9.11]{bridson-haefliger}).  We shall write $\xi = \cos\theta \xi_1 + \sin\theta \xi_2$ with $\theta \in [0,\pi/2]$ for a point of $\bTX$, where $\xi_1 \in \bTX_1$ and $\xi_2 \in \bTX_2$.  If $\xi = \xi_1+ 0$ (respectively $\xi = 0 + \xi_2$), then $C(\xi)$ lies in all the walls of $W_2$ (respectively all the walls of $W_1$), and $C(\xi) = \bTXi$ for an infinite reduced word $\ii$ of the form (1) (respectively  (2)).  A point of the form $\xi = \cos\theta \xi_1 + \sin\theta \xi_2$ with $\theta \in (0,\pi/2)$ has equivalence class given by
$$
C(\xi) = \{\cos\theta \xi'_1 + \sin\theta \xi'_2 \mid \xi'_1 \in C(\xi_1) \text{ and } \xi'_2 \in C(\xi_2)\}
$$
and we then have $C(\xi) = \bTXi$ for an infinite reduced word $\ii$ of the form (3).

If $W_2$ is finite (hence $W_1$ is infinite), then only case (1) can occur, and similarly if $W_1$ is finite only case (2) can occur.  In particular, it follows from \cite[Theorem 12.3.5]{davis-book} that a Coxeter group is virtually abelian if and only if it is a direct product of finite and affine Coxeter groups.  

It is thus straightforward to extend the results for irreducible affine Coxeter groups described in Section~\ref{s:affine} to all affine and all virtually abelian Coxeter groups.

\subsection{Coxeter groups with two ends}\label{s:2 ends}

Recall from Theorem \ref{T:davisends}(3) above that it is equivalent for $W$ to have two ends and for $W$ to be reducible of the form $W_0 \times W_1$ where $W_0 = \la s, t \ra$ is the infinite dihedral group and $W_1$ is spherical.  Any infinite reduced word in $W$ is then obtained by inserting into either $ststst\cdots$ or $tststs\cdots$ the letters of a reduced word for some $w_1 \in W_1$.  The boundary of $X$ thus consists of two points, say $\xi_+$  and $\xi_-$.  The only walls of $X$ with nonempty boundary are the finitely many walls corresponding to the reflections in $W_1$, and for each such wall $M$ we have $\bM = \{ \xi_+, \xi_-\}$.  In other words, $R(\xi_\pm)$ is equal to the set of reflections in $W_1$, and $W(\xi_\pm) = W_1$. Clearly $C(\xi_+) = \{ \xi_+ \}$ and $C(\xi_-) = \{ \xi_- \}$.  There are $2|W_1|$ equivalence classes of infinite reduced words, with $\ii \sim \jj$ if and only if $\ii$ and $\jj$ project to the same infinite reduced word in $W_0$ and $\ii$ and $\jj$ are eventually on the same side of every wall in $W_1$ (equivalently, the projections of $\ii$ and $\jj$ to $W_1$ define the same group element).  The partial order on the set of infinite reduced words $\ii$ with $\bTX(\ii) = \{ \xi_+\}$ is then just the weak partial order on $W_1$, and similarly for $\xi_-$.

\subsection{Coxeter groups with infinitely many ends}\label{s:infinite ends}

Assume now that $W$ has infinitely many ends.
In particular, the case that $W$ is virtually free and the case that $W$ is a free product of special subgroups are contained in the case that $W$ has infinitely many ends.  In this section we will use some results about graphs of groups, a reference for which is Chapter I of \cite{Serre}.

By \cite[Proposition 8.8.2]{davis-book}, every Coxeter group $W$ has a tree-of-groups decomposition (not in general unique) in which every vertex group is a spherical or a one-ended special subgroup, and every edge group is a spherical special subgroup.  More precisely, the edge groups are spherical special subgroups $W_T$ where the punctured nerve $L - \sigma_T$ is disconnected.  Since $W$ has infinitely many ends, by Theorem \ref{T:davisends} any such tree-of-groups decomposition has at least one edge.  Our analysis will depend upon the nature of the vertex and edge groups in such a decomposition.

First, the group $W$ is virtually free if and only if all vertex groups in such a tree-of-groups decomposition of $W$ are spherical \cite[Proposition 8.8.5]{davis-book}.  The Davis complex $X$ is then quasi-isometric to a tree, with $\bTX$ homeomorphic to the set of ends of this tree.  The topology on $\bTX$ is thus totally disconnected and so for each $\xi \in \bTX$, we have $C(\xi) = \{ \xi \}$.

Next, the group $W$ is a free product of spherical special subgroups if and only if all vertex groups in such a tree-of-groups decomposition of $W$ are spherical \emph{and} all edge groups are trivial.  The Davis complex $X$ then consists of infinitely many copies of the (finite) Davis complexes for the vertex groups, glued together at certain centres of chambers.  Thus all walls are finite subcomplexes of $X$ and so have empty boundary.  It follows that for each $\xi \in \bTX$, there is a unique equivalence class of infinite reduced words $\ii$ such that $\bTXi = C(\xi) = \{ \xi \}$ is a single point.

We next discuss the subcase that all vertex groups are spherical and there is some nontrivial edge group.  Fix a tree-of-groups decomposition of $W$ and let $r$ be a reflection in $W$.  Since reflections have finite order, each reflection is contained in a conjugate of at least one vertex group.  If $r$ is not contained in any conjugates of edge groups, then the wall $M = M(r)$ is contained in the finite Davis complex for a unique (conjugate of a) vertex group, and so $M$ has empty boundary.  If $r$ is contained in some (conjugate of an) edge group, then the wall $M = M(r)$ may or may not have empty boundary, as shown by the following examples.  (We do not provide a complete characterisation of the reflections $r$ such that the corresponding wall $M$ has nonempty boundary.)
\begin{examples}
\begin{enumerate}
\item Consider $W = \la s,t,u \mid s^2 = t^2 = u^2 = (st)^3 = (tu)^3 = 1 \ra$.  Then $W$ splits as the amalgamated free product $W = W_{\{s,t\}} *_{W_{ \{ t \} }} W_{\{ t,u\}} \cong S_3 *_{C_2} S_3$ and hence has a tree-of-groups decomposition consisting of an edge with vertex groups $W_{\{s,u\}}$ and $W_{\{t,u\}}$ and edge group $W_{\{t\}}$.  The reflection $t$ is obviously contained in an edge group, but the wall $M(t)$ is finite hence has empty boundary.  In fact all walls in this example have empty boundary.
\item Let $W = W_{T_1} *_{W_T} W_{T_2}$ where $T_1$ and $T_2$ are spherical and $T = T_1 \cap T_2$ is nonempty.  Let $r$ be a reflection in $W_T$.  Suppose that for $i = 1,2$ there is a reflection $r_i \in W_{T_i}\setminus W_T$ which commutes with $r$.  Then for $i = 1,2$ and every $w \in \la r_1, r_2 \ra \cong D_\infty$, the reflection $r$ is contained in $W_{T_i}^w$.  It follows from the fact that there are infinitely many distinct such conjugates $W_{T_i}^w$ and the construction of $X$ that the wall $M = M(r)$ has nonempty boundary, consisting of at least two points.  \end{enumerate}
\end{examples}

Continuing the analysis of the subcase in which all vertex groups are spherical and there is a non-trivial edge group $W_T$, let $M = M(r)$ be a wall with nonempty boundary, where the reflection $r$ is without loss of generality in $W_T$.  For each $\xi \in \bTX$ such that $M \in \Mxi$, there is then a nontrivial partial order on the set of equivalence classes of infinite reduced words $\ii$ such that $\bTXi = \xi$.  Since $X$ does not contain an isometrically embedded Euclidean plane, Corollary \ref{c:infinitexi} implies that  $\Mxi$ is finite.  Hence if $M' \neq M$ is also in $\Mxi$, with $r'$ the reflection fixing $M'$, then $r$ and $r'$ generate a finite subgroup of $W$.  Thus $\la r, r' \ra$ fixes a point $p$ in $X$, and we obtain a geodesic ray $[p,\xi)$ which is contained in both $M$ and $M'$.  It follows that $r'$ must be in $W_T$ as well.  Applying this argument to every wall in $\Mxi$, we conclude that every reflection in $\Rxi$ is contained in the edge group $W_T$.  There may however be reflections in $W_T$ which are not in $\Rxi$.

We finally consider the case that $W$ has infinitely many ends and is not virtually free, equivalently there is at least one one-ended vertex group in the tree-of-groups decomposition of $W$.  Then an infinite reduced word $\ii$ either eventually stays within a fixed conjugate $W_T^w$ of a one-ended vertex group $W_T$, or $\ii$ visits infinitely many distinct conjugates of vertex groups.  For each $w \in W$ and each one-ended vertex group $W_T$, the partial order on the set of equivalence classes of infinite reduced words that eventually stay in $W_T^w$ is the same as the partial order on the set of equivalence classes of infinite reduced words in $W_T$.  In a similar manner to the analysis when all vertex groups are spherical, if there is some nontrivial edge group, then there may also be a nontrivial partial order on equivalence classes of certain infinite reduced words $\ii$ which visit infinitely many distinct conjugates of vertex groups.

\subsection{Word hyperbolic Coxeter groups}\label{s:hyperbolic}

In this section we consider the important special case that $W$ is word hyperbolic.  In particular, we prove Theorem~\ref{thm:hyp}, stated in the introduction.  See \cite{bridson-haefliger} for the definition and basic properties of $\delta$-hyperbolic spaces and word hyperbolic groups.

We will use the following characterisation of word hyperbolic Coxeter groups, which is due to Moussong, and which appears in \cite[Corollary 12.6.3]{davis-book}.  

\begin{theorem}[Moussong]\label{t:hyperbolic} A Coxeter group $(W,S)$ is word hyperbolic if and only if there is no subset $T$ of $S$ satisfying either of the two conditions:
\begin{enumerate}
\item $W_T$ is affine of rank at least 2.
\item $(W_T,T)$ decomposes as $(W_T,T) = (W_{T'} \times W_{T''}, T' \cup T'')$ with both $W_{T'}$ and $W_{T''}$ infinite.
\end{enumerate}
\end{theorem}

We now characterise the word hyperbolicity of $W$ in terms of the sets $\Mxi$ in Proposition~\ref{p:W hyp locally finite}, in terms of the groups $\Wxi$ in Corollary~\ref{c:Wxi finite} and finally in terms of the sets $C(\xi)$ in Proposition~\ref{p:away from walls}.

\begin{proposition}\label{p:W hyp locally finite}  $W$ is word hyperbolic if and only if for all $\xi \in \bTX$, the set $\Mxi$ is finite.
\end{proposition}

\begin{proof}  Suppose there is a $\xi \in \bTX$ such that $\Mxi$ is infinite.  Then by Corollary \ref{c:Mxi infinite} and Proposition~\ref{p:flat}, the space $X$ contains an isometrically embedded Euclidean plane.  By the Flat Plane Theorem \cite[III.H.1.5]{bridson-haefliger}, this implies $X$ is not $\delta$-hyperbolic, hence $W$ is not word hyperbolic.

For the converse, let $T$ be a subset of $S$ satisfying either of the conditions in Theorem \ref{t:hyperbolic}.  We denote by $X_T$ the Davis complex for $(W_T,T)$, which embeds isometrically in the Davis complex $X$ for $(W,S)$.  It suffices to prove there is a $\xi$ in the Tits boundary of $X_T$ such that $\xi$ is in the boundary of infinitely many walls in $X_T$.

This is immediate when $W_T$ is affine, by taking $\xi$ in the boundary of a family of parallel walls in $X_T$.  If $W_T = W_{T'} \times W_{T''}$ as in (2) of Theorem \ref{t:hyperbolic}, let $\xi$ be any point in the boundary of $X_{T''} \subset X_T$ (since $W_{T''}$ is infinite, $X_{T''}$ has nonempty boundary).  Now as $W_{T'}$ is infinite, $W_{T'}$ contains infinitely many distinct reflections $r_i$.  The $r_i$ are also reflections in $W_T$, with corresponding walls say $M_i$ in $X_T$. Since $W_{T''}$ commutes with $W_{T'}$, the point $\xi$ is fixed by each of the reflections $r_i$, and so $\xi \in \bM_i$ for all $i$.
\end{proof}

\begin{corollary}\label{c:Wxi finite}  $W$ is word hyperbolic if and only if for all $\xi \in \bTX$, the group $\Wxi$ is finite.
\end{corollary}

\begin{proof}  This follows from Corollary~\ref{c:Rxi finite} and Proposition~\ref{p:W hyp locally finite}.
\end{proof}

\begin{proposition}\label{p:away from walls} $W$ is word hyperbolic if and only if for all $\xi \in \bTX$ the equivalence class $C(\xi)$ is reduced to~$\{ \xi \}$.
\end{proposition}

\begin{proof} Assume $W$ is word hyperbolic. Since $W$ is quasi-isometric to the Davis complex $X$ (equipped with its piecewise Euclidean metric), we have that $X$ is $\delta$--hyperbolic for some $\delta > 0$.  Suppose that there are $\xi \neq \xi' \in \bTX$ with $\xi' \in C(\xi)$.  Fix a basepoint $p \in X$. Then by Proposition \ref{p:partition} and Lemma \ref{c:same walls}, the geodesic rays $c = [p,\xi)$ and $c' = [p,\xi')$ cross the same set of walls.

For each wall $M$ crossed by $c$ (and thus by $c'$), denote by $\gamma_M$ the geodesic segment in $M$ which connects the points of intersection of $c$ and $c'$ with $M$, and let $\Delta_M$ be the geodesic triangle in $X$ with one vertex at $p$, with $\gamma_M$ its side not containing $p$, and with its other two sides initial subsegments of $c$ and $c'$.  Then by $\delta$--hyperbolicity of $X$, the geodesic segment $\gamma_M$ is contained in the $\delta$--neighbourhood of the union of $c$ and $c'$.

Since $\xi \neq \xi'$, the $\delta$--neighbourhoods of $c$ and $c'$ have bounded intersection.  Thus the geodesic segment $\gamma_M \subset M$ passes within bounded distance of $p$.  By the local finiteness of the collection of walls, there can be only finitely many distinct walls $M$ such that a geodesic segment contained in this wall is within bounded distance of $p$.  But $c$ (and thus $c'$) crosses infinitely many walls $M$, so we obtain a contradiction.  Thus $c$ and $c'$ are equivalent, and moreover their distance from each other is bounded in terms of $\delta$.  Hence $\xi = \xi'$.

Now assume that $W$ is not word hyperbolic, and let $W_T$ be an affine or reducible special subgroup which is an obstruction to the hyperbolicity of $W$, as in Theorem \ref{t:hyperbolic} above.  If $C(\xi) = \{ \xi \}$ then the closure $\overline{C(\xi)}$ equals $C(\xi)$, so by Proposition \ref{p:partition} and Theorem \ref{t:closure partition} above, it suffices to find infinite reduced words $\ii$ and $\jj$ such that $\jj < \ii$ (strict inequality) and $\bTX(\ii) \neq \bTX(\jj)$.  The existence of suitable pairs $\ii$ and $\jj$ in such $W_T$ follows from our discussion of the affine and reducible cases above.
\end{proof}

Since each set $\bTXi$ is an equivalence class $C(\xi)$, we have established the characterisation of word hyperbolicity of $W$ in Theorem~\ref{thm:hyp}.  Now if $W$ is hyperbolic, then as each $\bTXi$ is a singleton, we have $\overline{\bTXi} = \bTXi$ for each infinite reduced word $\ii$.  Hence by Theorem~\ref{thm:main}(1) and (4), infinite reduced words $\ii$ and $\jj$ are comparable only if they are are in the same block.  The final claim of Theorem~\ref{thm:hyp} then follows from Theorem~\ref{thm:Wxi intro}.  This completes the proof of Theorem~\ref{thm:hyp}.

\begin{example}\label{eg:3d} Let $P$ be a regular dodecahedron in $3$-dimensional (real) hyperbolic space with all dihedral angles right angles, and let $W$ be the Coxeter group generated by reflections in the codimension one faces of $P$.  Then $W$ is word hyperbolic and the Davis complex $X$ may be equipped with a piecewise hyperbolic metric so that $X$ is isometric to the induced tessellation of hyperbolic $3$-space by copies of $P$.  This tessellation is depicted in Figure \ref{fig:dodecahedrons}.

\begin{figure}
\begin{center}
\begin{overpic}[scale=0.3]{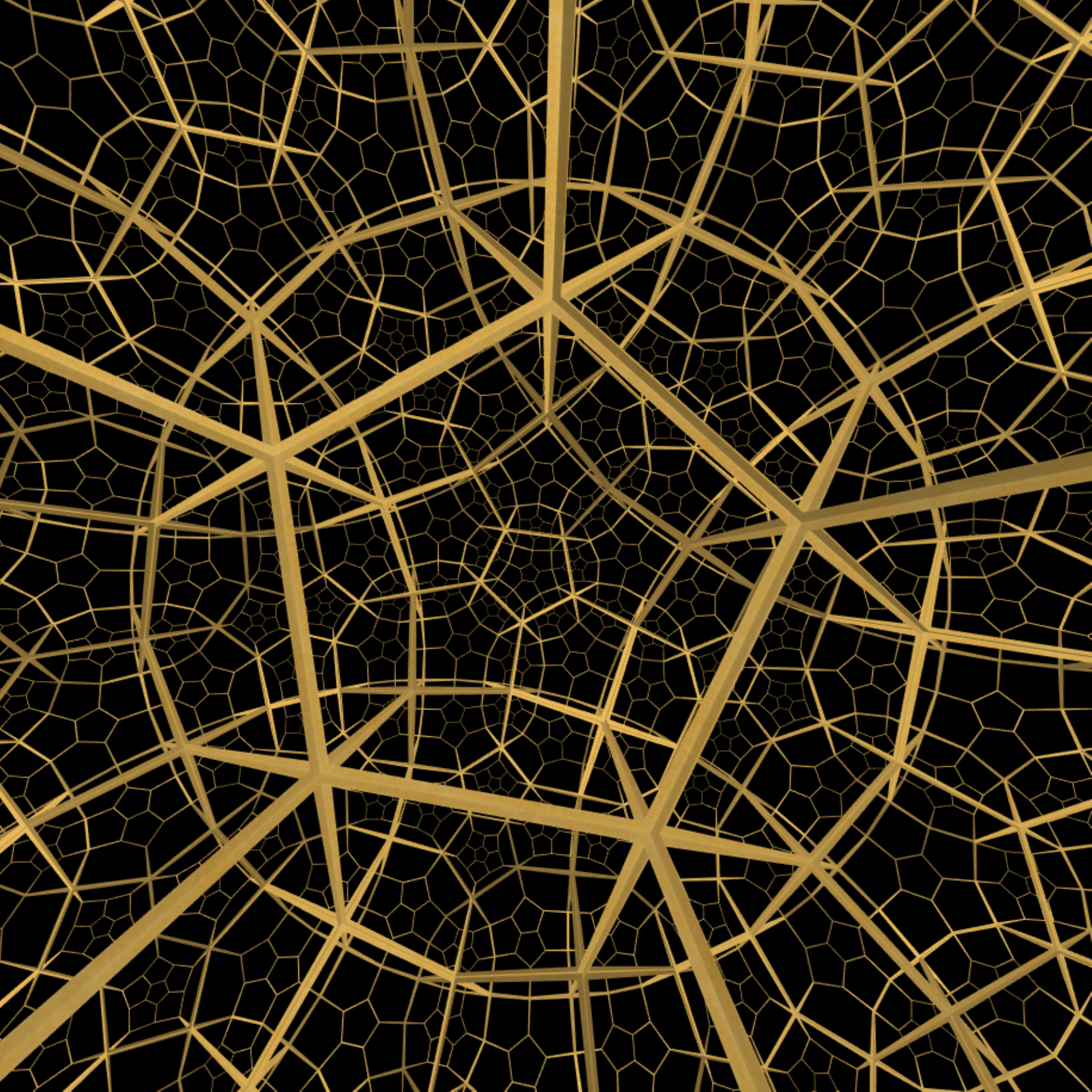}
\end{overpic}
\end{center}
\caption{The Davis complex for Example \ref{eg:3d}}
\label{fig:dodecahedrons}
\end{figure}

The Tits boundary $\bTX$ is a $2$-dimensional sphere and the distance between any two points in $\bTX$ is infinite.  Each wall of $X$ is a hyperplane tessellated by right-angled pentagons, and may be viewed as a copy of the Davis complex from Example \ref{eg:2d} in the introduction.  Each $\xi \in \bTX$ is thus contained in the boundaries of at most $2$ walls of $X$, with these walls perpendicular to each other.

Suppose $\ii$ is an infinite reduced word with $\xi = \xi(\ii)$.  If $\xi$ is in the boundary of two walls $M = M(r)$ and $M' = M'(r')$ then the block $B(\ii)$ consists of four equivalence classes of infinite reduced words, corresponding to the four elements of $W(\xi) = \langle r, r' \rangle \cong C_2 \times C_2$.  If $\xi$ is in the boundary of a unique wall $M = M(r)$ then the block $B(\ii)$ consists of two equivalence classes of infinite reduced words, corresponding to the two elements of the finite Coxeter group $W(\xi(\ii)) = \langle r \rangle$.  If $\xi$ is not in the boundary of any wall, then $B(\ii)$ contains a single equivalence class of infinite reduced words and $W(\xi(\ii))$ is trivial.
\end{example}

Given $\xi \in \bTX$, define
\[ \bMxi := \{ \bM \mid \xi \in \bM \} = \{ \bM \mid M \in \Mxi \}. \]
That is, $\bMxi$ is the set of boundaries of walls which contain the point $\xi$.  It is immediate that for all $\xi \in \bM$, $|\bMxi| \leq |\Mxi|$.

\begin{proposition}\label{p:bMxi finite}  Suppose $W$ is word hyperbolic and one-ended.  Then for all $\xi \in \bTX$, we have $|\bMxi| = |\Mxi| < \infty$.  \end{proposition}

\begin{proof}  Since $|\bMxi| \leq |\Mxi| < \infty$, it suffices to show that if $M, M'$ are distinct walls in $\Mxi$ then $\bM \neq \bM'$.  By Lemma \ref{l:Mxi intersection}, the intersection $M \cap M'$ is nonempty.  By Proposition \ref{p:geodesic extension}, since $W$ is one-ended $X$ has the geodesic extension property. Thus by Lemma \ref{l:distinct bM}, $\bM \neq \bM'$.
\end{proof}

We do not know how to characterise those $W$ such that $\bMxi$ is finite for all $\xi \in \bTX$ (noting that, by results above, if $W$ is affine or word hyperbolic then all $\bMxi$ are finite).  

%
%


\subsection{Isolated flats}\label{s:isolated}

In this section we use results of Caprace \cite{caprace_isolated} to discuss some $W$ for which $X$ is a space with isolated flats, in the sense studied by Hruska--Kleiner in \cite{hruska-kleiner}.  A \emph{flat} in $X$ is defined to be a subset which is isometric to Euclidean space of dimension $\geq 2$.

Given $T \subset S$ and a chamber $wK$ of $X$, the \emph{residue of type $T$} containing $wK$ is the subcomplex of $X$ consisting of all chambers $w'K$ such that $w^{-1}w' \in W_T$.  We denote by $R_T$ the residue of type $T$ containing the base chamber $K$.   The stabiliser in $W$ of the residue $R_T$ is $W_T$, and for any $w \in W$ the stabiliser of $wR_T$ is $W_T^w = wW_Tw^{-1}$.  Therefore every parabolic subgroup of $W$ stabilises some collection of residues.

By the construction and metrisation of $X$, there is a canonical isometric embedding of the Davis complex for $(W_T,T)$, denoted $X_T$, into $R_T$.  (In \cite{caprace_isolated}, the sets $X_T$ and $R_T$ appear to have been identified, although in \cite{caprace-haglund} and elsewhere a residue is defined as a union of chambers of $X$.) Hence for every $w \in W$, there is a canonical isometrically embedded copy of $X_T$ in the residue $wR_T$.  The action of $W_T^w$ on $wR_T$ preserves this copy of $X_T$.

Assume that $W$ is not word hyperbolic.  Then by Theorem \ref{t:hyperbolic} above, either there is some $T \subset S$ such that $W_T$ is affine of rank $\geq 2$, or there are non-spherical $T', T'' \subset S$ such that $[T',T''] = 1$.  If $W_T$ is affine of rank $\geq 2$ then $X_T$ is isometric to Euclidean space of dimension $\geq 2$.  Hence if $W_T$ is affine of rank $\geq 2$, then every residue  $wR_T$ contains a canonical flat which is preserved by the action of the affine parabolic subgroup $W_T^w$.

Let $\cT$ be the collection of all maximal subsets $T \subset S$ such that $W_T$ is affine (possibly reducible) and of rank $\geq 2$.  We make the following additional assumptions on the set $\cT$.
\begin{itemize}
\item[(RH1)] For each pair of irreducible non-spherical subsets $T', T'' \subset S$ such that $[T', T''] = 1$, there exists $T \in \cT$ such that $T' \cup T'' \subset T$.
\item[(RH2)] For all $T', T'' \in \cT$ with $T' \neq T''$, the intersection $T' \cap T''$ is spherical.
\end{itemize}
Since $W$ is not word hyperbolic, by (RH1) the set $\cT$ is nonempty.

As in the proof of Theorem A in \cite{caprace_isolated}, for each $T \in \cT$, denote by $N(W_T)$ the normaliser of $W_T$ in $W$.  Now define $\cR$ to be the collection of all residues $wR_T$ where $w$ runs through a transversal for $N(W_T)$ in $W$.  Note that we are considering representatives $w$ for the cosets of $N(W_T)$ in $W$, not for the cosets of $W_T$ in $W$, and so by Proposition 5.5 of \cite{deodhar_root} the set $\cR$ will \emph{not} contain all residues of types $T \in \cT$.  Let $\cF$ be the collection of canonical flats contained in the residues in $\cR$.

By Theorem A, Lemma 4.2 and Corollary D of \cite{caprace_isolated}, the Davis complex $X$ then has \emph{isolated flats}, meaning that the following conditions hold:
\begin{enumerate}
\item\label{i:flat in D nbhd} there is a constant $D < \infty$ such that every flat of $X$ is in the $D$--neighbourhood of some element of $\cF$; and
\item\label{i:bounded intersection} for each $0 < \rho < \infty$, there is a constant $\kappa = \kappa(\rho) < \infty$ so that for any two distinct flats $F, F' \in \cF$, the intersection of the $\rho$--neighbourhoods of $F$ and $F'$ has diameter $< \kappa$.
\end{enumerate}
(Also, $W$ is relatively hyperbolic with respect to its collection of maximal affine parabolic subgroups.)

By Theorem 1.2.1 of \cite{hruska-kleiner}, since $X$ has isolated flats the Tits boundary $\bTX$ is a disjoint union of isolated points and standard Euclidean spheres.  Note that by condition \eqref{i:flat in D nbhd} in the definition of isolated flats, for every flat $F$ in $X$ there is a flat $F' \in \cF$ such that $\partial_T F \subset \partial_T F'$.  By Theorem 5.2.5 of \cite{hruska-kleiner}, the spherical components of $\bTX$ are precisely the Tits boundaries of the flats in $\cF$.  We now analyse these spherical components.

For any flat $F$ in $X$, define $\MF$ to be the set of walls which separate points in $F$, and let $W(\MF))$ be the subgroup of $W$ generated by the set of reflections in walls in $\MF$.  Then by Corollary 3.2 of \cite{caprace_isolated}, the group $P:=\Pc(W(\MF))$ is a direct product of irreducible affine Coxeter groups.  (Here, $\Pc(W(\MF))$ denotes the parabolic closure of the group $W(\MF)$.)  Moreover, by Proposition 3.3 of \cite{caprace_isolated}, the flat $F$ is contained in some residue $R = wR_T$ whose stabiliser is $P$.   Thus in particular if $F \in \cF$ then by construction, $F$ is contained in a unique residue $R \in \cR$, and since the stabiliser of each residue in $\cR$ is a maximal affine parabolic subgroup of $W$ we have that $P = \Pc(W(\MF))$ is a maximal affine parabolic.  Further, since $P$ preserves $R$, the canonical flat $F \subset R$ intersects all of the walls corresponding to the reflections in $P$, and so $P = \Pc(W(\MF)) = W(\MF)$.

\begin{proposition}  Let $R,R'$ be any two residues in $\cR$ which have stabiliser the same maximal affine parabolic $P$, and let $F,F'$ be the corresponding flats in $\cF$.  Then:
\begin{enumerate}
\item\label{i:same flat} $F = F'$, hence $R = R'$ is the unique residue in  $\cR$ which is stabilised by $P$;
\item\label{i:MF} $\MF = \mathcal{M}(F')$ induce the same arrangement of boundaries of walls in $\partial_T F$; and
\item\label{i:wall isolated} if $M$ is any wall of $X$ such that $M \not \in \MF$ and $\bM$ intersects $\partial_T F$, then $\bM$ contains $\partial_T F$.
\end{enumerate}
\end{proposition}

\begin{proof}  By condition \eqref{i:bounded intersection} in the definition of isolated flats above, to show that $F = F'$ it suffices to show that for some $0 < \rho < \infty$, the intersection of the $\rho$--neighbourhoods of $F$ and $F'$ is unbounded.  Using the same argument as in the last paragraph of the proof of Proposition 3.3 of \cite{caprace_isolated}, we may obtain $R'$ from $R$ by applying a sequence of reflections which commute with $P$.  So we may without loss of generality assume that $R' = r.R$ for some reflection $r$ which commutes with $P$.  Fix a point $x_0 \in F$ and let $x_0' = r.x_0 \in F'$.  Choose a finite $\rho$ so that $\rho > d(x_0,x_0')$ and let $w$ be an element of infinite order in $P$.  Then for each $n \geq 1$, we have
$$\rho > d(x_0,x_0') = d(w^n.x_0', w^n.x_0) = d(w^n.r.x_0, t^n.x_0) = d(r.(w^n.x_0), w^n.x_0).$$
Since the set $\{w^n.x_0 \mid n \geq 1 \}$ is unbounded, we conclude that $F = F'$.

Part \eqref{i:MF} of this result is a formal consequence of part \eqref{i:same flat}.   For \eqref{i:wall isolated}, suppose $M$ is a wall such that $M \not \in \MF$, $\bM$ intersects $\partial_T F$ but $\bM$ does not contain $\partial_T F$.  Then $\bM$ must separate two points say $\xi$ and $\xi'$ of the sphere $\partial_T F$.  Let $c$ and $c'$ be geodesic rays from $x_0 \in R$ such that $c(\infty) = \xi$ and $c'(\infty) = \xi'$.  Since $M \not \in \mathcal{M}(F)$, the rays $c$ and $c'$ lie on the same side of $M$.  But then
their limit points $\xi$ and $\xi'$ must lie on the same side of $\bM$, a
contradiction.
\end{proof}

Thus we have that for each maximal flat $F$ in $X$ there is a (unique) flat $F' \in \cF$ such that $\partial_T F = \partial_T F'$ is a sphere component of the boundary $\bTX$, and the arrangement of boundaries of walls in this sphere component is precisely that induced by $\MF$.

On the other hand, if $\xi \in \bTX$ is an isolated point then $C(\xi) = \{ \xi \}$, since the $C(\xi)$ are connected, and we can say little else.

\subsection{Acknowledgements}

We thank the University of Sydney for travel support, \'Akos Seress for a helpful suggestion and Martin Bridson for helpful conversations.  We also thank Pierre-Emmanuel Caprace and Jean L\'ecureux for bringing to our attention the references \cite{Caprace-Lecureux}, \cite{hagen} and \cite{niblo-reeves}.  The first author thanks Pasha Pylyavskyy for collaboration which led to the line of thinking in this work.  Figures \ref{fig:pentagons} and \ref{fig:dodecahedrons} were produced using Jeff Weeks' free software KaleidoTile and CurvedSpaces, respectively.  We thank an anonymous referee for thorough reading and many helpful suggestions.

T.L. was supported by NSF grants DMS-0901111 and DMS-1160726, and by a
Sloan Fellowship.  A.T. was supported by ARC grant DP110100440 and by an Australian Postdoctoral Fellowship.

\bibliographystyle{siam}
\bibliography{refs}

\begin{thebibliography}{10}

\bibitem{Bjorner-Brenti}
{\sc A.~Bj{{\"o}}rner and F.~Brenti}, {\em Combinatorics of {C}oxeter groups},
  vol.~231 of Graduate Texts in Mathematics, Springer, New York, 2005.

\bibitem{bridson-haefliger}
{\sc M.~R. Bridson and A.~Haefliger}, {\em Metric spaces of non-positive
  curvature}, vol.~319 of Grundlehren der Mathematischen Wissenschaften
  [Fundamental Principles of Mathematical Sciences], Springer-Verlag, Berlin,
  1999.

\bibitem{brink-howlett}
{\sc B.~Brink and R.~B. Howlett}, {\em A finiteness property and an automatic
  structure for {C}oxeter groups}, Math. Ann., 296 (1993), pp.~179--190.

\bibitem{caprace_isolated}
{\sc P.-E. Caprace}, {\em Buildings with isolated subspaces and relatively
  hyperbolic {C}oxeter groups}, Innov. Incidence Geom., 10 (2009), pp.~15--31.

\bibitem{caprace-haglund}
{\sc P.-E. Caprace and F.~Haglund}, {\em On geometric flats in the {CAT}(0)
  realization of {C}oxeter groups and {T}its buildings}, Canad. J. Math., 61
  (2009), pp.~740--761.

\bibitem{Caprace-Lecureux}
{\sc P.-E. Caprace and J.~L{\'e}cureux}, {\em Combinatorial and group-theoretic
  compactifications of buildings}, Ann. Inst. Fourier (Grenoble), 61 (2011),
  pp.~619--672.

\bibitem{Cellini-Papi}
{\sc P.~Cellini and P.~Papi}, {\em The structure of total reflection orders in
  affine root systems}, J. Algebra, 205 (1998), pp.~207--226.

\bibitem{davis-book}
{\sc M.~W. Davis}, {\em The geometry and topology of {C}oxeter groups}, vol.~32
  of London Mathematical Society Monographs Series, Princeton University Press,
  Princeton, NJ, 2008.

\bibitem{deodhar_root}
{\sc V.~V. Deodhar}, {\em On the root system of a {C}oxeter group}, Comm.
  Algebra, 10 (1982), pp.~611--630.

\bibitem{Deodhar}
\leavevmode\vrule height 2pt depth -1.6pt width 23pt, {\em A note on subgroups
  generated by reflections in {C}oxeter groups}, Arch. Math. (Basel), 53
  (1989), pp.~543--546.

\bibitem{Dyer}
{\sc M.~Dyer}, {\em Reflection subgroups of {C}oxeter systems}, J. Algebra, 135
  (1990), pp.~57--73.

\bibitem{hagen}
{\sc M.~F. Hagen}, {\em The simplicial boundary of a {CAT}(0) cube complex},
  Algebr. Geom. Topol., 13 (2013), pp.~1299--1367.

\bibitem{hruska-kleiner}
{\sc G.~C. Hruska and B.~Kleiner}, {\em Hadamard spaces with isolated flats},
  Geom. Topol., 9 (2005), pp.~1501--1538.
\newblock With an appendix by the authors and Mohamad Hindawi.

\bibitem{Ito}
{\sc K.~Ito}, {\em Parameterizations of infinite biconvex sets in affine root
  systems}, Hiroshima Math. J., 35 (2005), pp.~425--451.

\bibitem{Lam-Pylyavskyy}
{\sc T.~Lam and P.~Pylyavskyy}, {\em Total positivity for loop groups {II}:
  {C}hevalley generators}, Transform. Groups, 18 (2013), pp.~179--231.

\bibitem{niblo-reeves}
{\sc G.~A. Niblo and L.~D. Reeves}, {\em Coxeter groups act on {${\rm CAT}(0)$}
  cube complexes}, J. Group Theory, 6 (2003), pp.~399--413.

\bibitem{Serre}
{\sc J.-P. Serre}, {\em Trees}, Springer Monographs in Mathematics,
  Springer-Verlag, Berlin, 2003.
\newblock Translated from the French original by John Stillwell, Corrected 2nd
  printing of the 1980 English translation.

\end{thebibliography}

\end{document}